\documentclass[a4paper,11pt]{article}

\usepackage[english]{babel}

\usepackage{pdfpages}
\usepackage{amsthm,amsmath,amssymb,amsfonts,graphicx,shorttoc,textpos,caption,here}
\usepackage{mathtools}
\numberwithin{equation}{section}
\usepackage[T1]{fontenc}
\usepackage[utf8]{inputenc}
\usepackage{multirow}
\usepackage{subcaption}

\usepackage{hyperref}
\makeatletter
\newcommand{\labitem}[2]{%
\def\@itemlabel{\textbf{#1}}
\item
\def\@currentlabel{\textbf{#1}}\label{#2}}
\makeatother

\usepackage{verbatim,enumerate,dsfont,fancyhdr,setspace,listings,lscape,eurosym}

\usepackage{xcolor}
\usepackage[numbers]{natbib}
\usepackage{graphicx}
\graphicspath{{illustrations/}}
\usepackage{xcolor}
\usepackage{array}

\usepackage{multirow}
\usepackage{calc}
\usepackage[normalem]{ulem}
\usepackage{textpos}
\usepackage{csquotes}
\usepackage{longtable}
\setlongtables

\newcommand{\Aa}{\mathcal{A}}
\newcommand{\Bb}{\mathcal{B}}

\newcommand{\Dd}{\mathcal{D}}
\newcommand{\Ee}{\mathcal{E}}
\newcommand{\EE}{\mathbb{E}}

\newcommand{\Hh}{\mathcal{H}}

\newcommand{\Kk}{\mathcal{K}}

\newcommand{\LL}{\mathbb{L}}

\newcommand{\Nn}{\mathcal{N}}
\newcommand{\NN}{\mathbb{N}}

\newcommand{\PP}{\mathbb{P}}

\newcommand{\RR}{\mathbb{R}}

\newcommand{\Ss}{\mathcal{S}}

\newcommand{\Ww}{\mathcal{W}}

\newcommand{\Zz}{\mathcal{Z}}

\newcommand{\1}{{\mathds 1}}

\newcommand{\var}{{\rm Var}}

\newcommand{\norm}[1]{\lVert #1 \rVert}
\newcommand{\lnorm}[1]{\left \lVert #1 \right \rVert}
\newcommand{\abs}[1]{\left\lvert #1 \right\rvert}
\newcommand{\dd}{{\rm d}}

\newcommand{\toP}{\overset{\mathbb{P}}{\to}}

\newcommand{\toDd}{\xrightarrow[]{\Dd}}

\newcommand{\Clemma}[1]{C_{M_t}^{#1}}

\newtheorem{theorem}{Theorem}%[section]
\newtheorem{lemma}[theorem]{Lemma}%[section]
\newtheorem{corollary}[theorem]{Corollary}%[section]
\newtheorem{proposition}[theorem]{Proposition}%[section]

\theoremstyle{definition}
\newtheorem{example}{Example}%[section]
\DeclareMathOperator{\K}{K}

\begin{document}
\vspace{-0.5cm}
\begin{center}

{\huge\bf Kernel estimation of the intensity of Cox processes\\
\vspace{0.2cm}}

\vspace{0.5cm}
Nicolas \textsc{Klutchnikoff}

IRMAR, Université de Rennes 2, CNRS, UEB\\
Campus Villejean (Rennes)\\
Place du recteur Henri Le Moal\\
CS 24307\\
35043 Rennes cedex\\
\smallskip
{nicolas.klutchnikoff.univ-rennes2.fr}

\vspace{0.5cm}
Gaspar \textsc{Massiot}

IRMAR, ENS Rennes, CNRS, UEB\\
Campus de Ker Lann\\
Avenue Robert Schuman,
35170 Bruz, France\\
\smallskip
{gaspar.massiot@univ-rennes1.fr}

\end{center}

\section{Introduction}

%Andersen1993,
Counting processes and in particular Cox processes have been used for many years to model a large variety  of situations from neuroscience \citep[see][]{Bialek1991,Brette2008,Krumin2009} to seismic \citep[see][]{Ogata1988}, financial \citep[see][]{Merton1976}, insurance \citep[see][]{Asmussen2010} or  biophysical data \citep[see][]{Kou2005a}. Recall that a Cox process $N=(N_t)_{t\in[0,1]}$ with random intensity $\lambda=\big(\lambda(t)\big)_{t\in[0,1]}$ is a counting process  such that the conditional distribution of $N$ given $\lambda$ is a Poisson process with intensity $\lambda$. In all the previous situations one of the main problem can be summarized as the estimation of the intensity $\lambda$ of the process \citep[see][]{Zhang2010}.

Note that when Cox process data arise the intensity of the process is mainly not directly observed but a co-process is observed instead. Returning to one of the previous example, in single-molecule experiments only the peaks inducing the counting process and an underlying process are observed \citep[see][]{Kou2005a}. Another example can be found in car insurance \citep[see][]{Asmussen2010} where the counting process models the occurrence of car crash that are subject to weather conditions.
In these cases the counting process $N=(N_t)_{t\in[0,1]}$ that naturally raises is accompanied with a co-process $Z=(Z_t)_{t\in[0,1]}$ such that the conditional law of $N$ given $Z$ is a Poisson process with intensity $\theta(t,Z)$ where $\theta$ is a deterministic function. By a slight abuse we shall call Cox process such a counting process. From a statistical point of view one of the major issue is to estimate the deterministic function $\theta$ using $n$ independent copies $(N^1,Z^1),\ldots,(N^n,Z^n)$ of $(N,Z)$. However, such an approach is  subject to the curse of dimensionality as the covariate $Z$ takes its values in an infinite dimension space as seen in \citet{OSullivan1993}.

When dealing with practical problems it is often unnecessary, or at least not strictly required for the modeling, to observe the full trajectory of the co-process. One can instead observe the values taken by the co-process at some well chosen random times that cover most of the information in the co-process. In this model the co-process is observed at a finite number of random times thereby circumventing the curse of dimensionality. 

In this paper we consider the following model: let $N=(N_t)_{t\in[0,1]}$ be a counting process and $Z=(Z_t)_{t\in[0,1]}$ be a $\RR^d$-valued co-process. We assume that $N$ admits a random intensity which depends on $t$ and on the observations of $Z$ at random times $S_1<S_2<\ldots$. 

More precisely, given the $\sigma$-algebra $\Ss$ generated by these times, $N$ is a Cox process with intensity
\begin{align}\label{eq:model}
\lambda(t,Z)=\theta_S\left(t,\vec Z_S(t)\right),
\end{align}
where $M$ is the counting process associated to $S=(S_1,S_2,\ldots)$, for any function $z:[0,1]\to\RR$, $\vec z_S(t)$ denotes the projection $(z_{S_1},\ldots,z_{S_{M_t}})\in\RR^{dM_t}$  and  $\theta_S(t,\cdot)$ is a function from $\RR^{dM_t}$ into $\RR_+$.

In the sequel we consider that given $\Ss$, $(N^1,Z^1),\ldots,(N^n,Z^n)$ are independent and identically distributed (\emph{i.i.d.}) copies of $(N,Z)$.  
The goal of this paper is to construct and study  the statistical properties of a kernel-type estimator of $\lambda$ using these data. Note that the dimension of our estimation problem, which depends on the counting process $(M_t)_{t\in[0,1]}$, increases with $t$. This potentially leads to a deterioration of the accuracy of any estimation procedure as the time variable increases.

We consider a substantial data set of historical prices of $495$ companies and the crude oil prices over a period of roughly one year and two months (from 17th April, 2014 to 23rd June, 2015). The Cox process data consist of the count of the number of times when the percent returns of said companies go below a certain threshold with the counting rate depending on the stochastic dynamics of the company market capitalization. In this example the company market capitalization is represented by the action's trade volume normalized increments and is observed when the percent return of the crude oil action goes below another threshold. By analyzing this count, we aim to learn the financial properties of this $495$ companies system.

The paper is organized as follows. Section~\ref{seq:ModelEstimate} presents the estimator we propose and its asymptotic properties. In Section~\ref{seq:Simu} we proceed to a simulation study. Then in Section~\ref{seq:RealData} we apply the proposed estimator on the real data set presented above. Technical proofs of the asymptotic properties are postponed to Section~\ref{seq:Proofs}.  

\section{Estimation strategy and results}\label{seq:ModelEstimate}

\subsection{Estimation strategy}

Let $t\in[0,1]$ and $z:[0,1]\to\RR$ be fixed. In this section we present the main ideas behind the construction of our estimator $\tilde \lambda(t,z)$ of $\lambda(t,z)$.

As an introduction to our methodology we consider the ideal case where we observe $\theta_S\left(t,\vec Z_S^k(t)\right)$ for all $k=1,\ldots,n$. Then our problem of estimation can simply be viewed as a regression estimation problem where $\lambda(t,\cdot)$ is the regression function. In this context, the Nadaraya-Watson estimator writes
\begin{align*}
\hat\lambda_{NW}(t,z)=\frac{\sum_{k=1}^n\theta_S\left(t,\vec Z_S^k(t)\right)H_\eta\left(\vec z_S(t)-\vec Z_S^k(t)\right)}{\sum_{l=1}^{n}H_{\eta}\left(\vec z_S(t)-\vec Z_S^{l}(t)\right)}.
\end{align*}
where $H_\eta$ denotes the multivariate product kernel $\Hh_\eta^{\otimes dM_t}$ where $\Hh$ is a kernel, that is $\Hh\in\LL^1(\RR)$ such that $\int_\RR \Hh(u)\dd u=1$, $\eta$ is an $\Ss$-measurable positive random variable (called a bandwidth) and $\Hh_\eta(\cdot)=\eta^{-1}\Hh(\eta^{-1}\cdot)$.
% Problème : introduction de l'estimateur de \theta_S(\cdot,\vec Z_S^k(t)) : Processus de Poisson, etc.

In practice $\theta_S\left(t,\vec Z_S^k(t)\right)$  can be estimated using the observations. Indeed, conditionally to $\Ss$ and the $\sigma$-algebra $\Zz$ generated by the co-processes $(Z^1,\ldots,Z^n)$, $N$ is a non-homogeneous Poisson process with intensity function $t\mapsto\theta_S \big(t,\vec Z_S^k(t)\big)$, a natural estimator of this intensity is given by
\begin{align*}
\int_0^t K_h(t-u)\dd N_u^k=\sum_{i=1}^{N_t^k} K_h(t-T_i^k),
\end{align*}
where $T_1^k,T_2^k,\ldots$ denote the jumping times of the trajectory $N^k$, $K:\RR_+\to\RR$ is a kernel and $h$ is a bandwidth. Denoting 
\begin{align*}
	\hat \phi_{S,h,\eta}\left(t,\vec z_S(t)\right)&=\frac{1}{n}\sum_{k=1}^{n}\sum_{i=1}^{N_{t}^{k}}K_{h}\left(t-T_{i}^{k}\right)H_{\eta}\left(\vec z_S(t)-\vec Z_S^{k}(t)\right),\\
	\hat f_{S,\eta}\big(\vec z_S(t)\big)&=\frac{1}{n}\sum_{l=1}^{n}H_{\eta}\left(\vec z_S(t)-\vec Z_S^{l}(t)\right),
\end{align*}
we define the plug-in estimator by
\begin{align*}
\hat{\lambda}\left(t,z\right)&=\frac{\hat \phi_{S,h,\eta}\big(t,\vec z_S(t)\big)}{\hat f_{S,\eta}\big(\vec z_S(t)\big)}.
\end{align*}

For the sake of stability \citep[see][]{Bickel1982} we consider a trimmed version of the previous estimator
\begin{align}\label{eq:estimator}
\tilde\lambda(t,z)=\tilde{\theta}_{S,h,\eta}\big(t,\vec z_S(t)\big)&=\frac{\hat \phi_{S,h,\eta}\big(t,\vec z_S(t)\big)}{\tilde f_{S,\eta}\big(\vec z_S(t)\big)},
\end{align} 
where $\tilde f_{S,\eta}\big(\vec z_S(t)\big)=\hat f_{S,\eta}\big(\vec z_S(t)\big)\vee a_n$ and $(a_n)_{n\in\NN}$ is an $\Ss$-measurable real-valued positive decreasing sequence.

\subsection{Results}\label{subsec:Results}

In this paper we are interested in the local behaviour of our estimator. We thus consider the pointwise mean squared error defined by
\begin{align}\label{eq:MSE}
{\rm MSE}(t,z)&=\EE\left[\left(\tilde\theta_{S,h,\eta}\big(t,\vec z_S(t)\big)-\theta_S\big(t,\vec z_S(t)\big)\right)^2\right],
\end{align}
and we make local regularity assumptions on the model. 

Remark that, almost surely, $t\in ]S_{M_t},S_{M_t}+1[=I_{M_t}$. For any $u\in I_{M_t},$ $\vec Z_S(u)=\vec Z_S(t)\in\RR^{dM_t}$ and $\theta_{S,M_t}(u,\cdot)=\theta_S(u,\cdot)$ is defined from $\RR^{dM_t}$ into $\RR_+$.
\begin{enumerate}
\labitem{(H1)}{hyp:existence} Given $\Ss$, for any $u\in I_{M_t}$, $\vec Z_S(u)$ admits a conditional density~$f_{M_t}$ defined from $\RR^{dM_t}$ into $\RR_+$;
\labitem{(H2)}{hyp:continuity} $\theta_{S,M_t}$ and $f_{M_t}$ are positive continuous functions;
\labitem{(H3)}{hyp:differentiability} $\theta_{S,M_t}$ and $f_{M_t}$ are twice differentiable and there exists a random variable $Q_{M_t}$ depending on the dimension $dM_t$ such that for all $1\leq k\leq dM_t$,  
\[\forall y\in\RR^{dM_t}, \lnorm{\frac{\partial^2}{\partial x_k^2}f_{M_t}(y)}_2\leq Q_{M_t},\]
and for all $1\leq k\leq dM_t+1$,
\[\forall (y,u)\in\RR^{dM_t}\times I_{M_t},\lnorm{\frac{\partial^2}{\partial x_k^2}\Big(f_{M_t}(y)\theta_{S,M_t}\big(u,\vec y_S(u)\big)\Big)}_2\leq Q_{M_t},\]
where $\norm{\cdot}_2$ is the euclidean norm;
\labitem{(H4)}{hyp:norm} There exists positive constants $F_0$, $F_\infty$ and $\Theta$ such that $\norm{\theta_{S,M_t}}_\infty <\Theta^{M_t}$ and $0<F_0^{M_t}\leq f_{M_t}\leq F_\infty^{M_t}<\infty.$
\end{enumerate} 

%\begin{remark} 
%A sufficient condition to get assumption {\bf (H2)} is that $\theta$ is a Lipschitz continuous function and $Z_t$ has a Lipschitz continuous density. 
%In a same way, a sufficient condition to get assumption {\bf (H3)} is that $\theta$ and $f_{Z_t}$ are twice differentiable and have Lipschitz continuous partial derivatives and $f_{Z_t}(y)\neq0$.
%\end{remark}
We also make technical assumptions on the kernels, the sequence $a_n$ and $Q_{M_t}$
\begin{enumerate}
\labitem{(H5)}{hyp:H} $\Hh$ is a kernel of order $2$ (that is, for all $j$ from $1$ to $2$, $\int_\RR u^j\Hh(u)\dd u=0$ and $\int_\RR \abs{\Hh(u)}\dd u<+\infty$), ${\rm supp} \Hh=[-1,1]$ and $\norm{\Hh}_\infty<\infty$ ;
\labitem{(H6)}{hyp:K} $K$ is a kernel of order $2$, $K\in\LL^4(\RR)$ and ${\rm supp} K=[0,1]$;
\labitem{(H7)}{hyp:an} $a_n=\left(n\eta^{dM_t}\right)^{\varepsilon-1}$ for some $\varepsilon\in(0,1/2)$;
\labitem{(H8)}{hyp:speedM} There exists a positive constant $Q_0(t)$ such that $\EE[M_t^8 Q_{M_t}^4]\leq Q_0(t)$. Moreover, for any $\lambda>0$, there exists a positive constant $Q_1(\lambda,t)$ such that  $\EE e^{\lambda M_t}<Q_1(\lambda,t)$.
%\item[(H6)] $\int ||y||^3K(y)\dd y<+\infty$.
%\item[(H9)] $\int ||y||^3H(y)\dd y<+\infty$.
\end{enumerate}

We are now in position to state our main results. Define the pointwise conditional mean squared error by
\begin{align*}
{\rm MSE}_\Ss(t,z)=\EE_\Ss\left[\left(\tilde\theta_{S,h,\eta}\big(t,\vec z_S(t)\big)-\theta_S\big(t,\vec z_S(t)\big)\right)^2\right].
\end{align*}

%%%%%%%%%%%%%%%%%%%%%%%%%%%%%%%%%%%%%%%%%%%%%%%%%%%%%%%%%%%%%%
%%%%%%%%%%%%%%%%%% Theorem 1 %%%%%%%%%%%%%%%%%%%%%%%%%%%%%%%%%
%%%%%%%%%%%%%%%%%%%%%%%%%%%%%%%%%%%%%%%%%%%%%%%%%%%%%%%%%%%%%%

\begin{theorem}\label{theo:MSEcond}
Assume that {\bf \ref{hyp:existence}} to {\bf\ref{hyp:speedM}} are satisfied. Let $h$ and $\eta$ be two $\Ss$-measurable bandwidths such that $h\to0$, $\eta\to 0$ $nh\eta^{dM_t}\to+\infty$ and $n\eta^{dM_t+4}\to0$ almost surely ({\normalfont a.s.}) as $n\to+\infty$ then the pointwise conditional mean squared error writes
\begin{align}\label{eq:theoMSEcond}
{\rm MSE}_\Ss(t,z)&= \underset{n\to+\infty}{O}\left(h^4+\eta^4 + \frac{1}{nh\eta^{dM_t}}\right),\text{ }\PP-a.s.
\end{align}
\end{theorem}

Note that the optimal choice of the $\Ss$-measurable bandwidths $h$ and $\eta$ for the pointwise conditional mean squared error is then $h=\eta=n^{-\frac{1}{5+dM_t}}$. This leads to the following corollary for the control of the pointwise mean squared error.

\begin{corollary}\label{cor:MSE}
Under the assumptions of theorem~\ref{theo:MSEcond}, the pointwise mean squared error writes
\begin{align}\label{eq:corMSE}
{\rm MSE}(t,z)&=\underset{n\to+\infty}{O}\left(\EE\left(n^{-\frac{4}{5+dM_t}}\right)\right),\text{ }\PP-a.s.
\end{align}
\end{corollary}

We can get the consistency of our estimator under weaker assumptions as shown in the following proposition.

\begin{proposition}\label{prop:CvProba}
Assume that {\bf \ref{hyp:existence}}, {\bf \ref{hyp:continuity}} and {\bf \ref{hyp:norm}} to {\bf\ref{hyp:an}} are satisfied. Let $h$ and $\eta$ be two $\Ss$-measurable bandwidths such that $h\to0$, $\eta\to 0$, $nh\eta^{dM_t}\to+\infty$ {\normalfont a.s.} as $n\to+\infty$, then
\begin{align*}
\tilde\theta_{S,h,\eta}\big(t,\vec z_S(t)\big)\toP
\theta_S\big(t,\vec z_S(t)\big).
\end{align*}
\end{proposition}

%%%%%%%%%%%%%%%%%%%%%%%%%%%%%%%%%%%%%%%%%%%%%%%%%%%%%%%%%%%%%%
%%%%%%%%%%%%%%%%%% Theorem 3 %%%%%%%%%%%%%%%%%%%%%%%%%%%%%%%%%
%%%%%%%%%%%%%%%%%%%%%%%%%%%%%%%%%%%%%%%%%%%%%%%%%%%%%%%%%%%%%%

Theorem~\ref{theo:CvDistr} shows the asymptotic normality of our estimator.

\begin{theorem}\label{theo:CvDistr}
Assume that {\bf \ref{hyp:existence}}  to {\bf\ref{hyp:speedM}} are satisfied. Let $h$ and $\eta$ be two $\Ss$-measurable bandwidths such that $h\to0$, $\eta\to 0$, $nh\eta^{dM_t}\to+\infty$, $nh\eta^{dM_t+4}\to0$, and $nh^5\eta^{dM_t}\to0$ {\normalfont a.s.} as $n\to+\infty$ then for any $z:[0,1]\to\RR$ such that $\theta_S\left(t,\vec z_S(t)\right)\neq 0$

\begin{align*}
\left(nh\eta^{dM_t}\right)^{1/2}\frac{\tilde\theta_{S,h,\eta} \big(t,\vec z_S(t)\big)-\theta_S \big(t,\vec z_S(t)\big)}{\left[\tilde \theta_{S,h,\eta} \big(t,\vec z_S(t)\big) \norm{K}_2^2 \norm{\Hh}_2^{2dM_t} / \tilde f_{S,\eta} \big(\vec z_S(t)\big)\right]^{1/2}}\toDd\Nn(0,1).
\end{align*}

\end{theorem}

\emph{Remarks.}
Corollary~\ref{cor:MSE} gives the tools to define optimal bandwidths in terms of pointwise asymptotic mean squared error. Assumptions must however be made on the process $M$ to conclude on the convergence rate since the ${\rm MSE}$ depends on the quantity $\EE \left(n^{-\frac{4}{5+dM_t}}\right)$. In what follows we assume that $M$ is a renewal process with inter-arrival times distributed accordingly to a strictly increasing cumulative distribution function $F$. The behaviour of the considered expectation is linked to the local behaviour of $F$ around $0$. The two following examples give incentive on the performances of our estimator for $F$ close to $0$ around $0$.
\begin{example}\label{example:1}
Let $\varepsilon$ be a positive constant and assume that $F(x)=0$ for any $x\in[0,\varepsilon]$. Then, 
\begin{align*}
\EE \left(n^{-\frac{4}{5+dM_t}}\right)&\leq n^{-\frac{4}{5}}\PP(S_0\leq t)+\sum_{k\geq 1}\left[n^{-\frac{4}{5+dk}}-n^{-\frac{4}{5+d(k-1)}}\right]\PP(S_k\leq t)\\
&\leq n^{-\frac{4}{5}}+\sum_{k= 1}^{\lfloor\frac{t}{\varepsilon}\rfloor}\left[n^{-\frac{4}{5+dk}}-n^{-\frac{4}{5+d(k-1)}}\right]\\
&\leq n^{-\frac{4}{5+d\lfloor\frac{t}{\varepsilon}\rfloor}}.
\end{align*}
So that \eqref{cor:MSE} gives
\begin{align*}
{\rm MSE}(t,z)=\underset{n\to+\infty}{O}\left(n^{-\frac{4}{5+d\left\lfloor\frac{t}{\varepsilon}\right\rfloor}}\right).
\end{align*}
This is the optimal rate of convergence for the nonparametric regression with a twice continuously differentiable regression function from $\RR^{d\left\lfloor\frac{t}{\varepsilon}\right\rfloor+1}$ to $\RR$~\citep[see][]{Gyorfi2002}.
\end{example}

\begin{example}
Let $\varepsilon$ and $\alpha$ be two positive constants such that $\alpha>1$ and assume that
$F(x)\leq\exp\left\{-(\varepsilon x^{-1})^\alpha\right\},$ around $0$. Then, 
\begin{align*}
\EE \left(a^{M_t}\right)&\leq n^{-\frac{4}{5}}\PP(S_0\leq t)+\sum_{k\geq 1}\left[n^{-\frac{4}{5+dk}}-n^{-\frac{4}{5+d(k-1)}}\right]\PP(S_k\leq t)\\
&\leq n^{-\frac{4}{5+dk^*}}+\sum_{k\geq k^*+1}kF\left(\frac{\varepsilon}{(3\log k)^{1/\alpha}}\right)\\
&\leq n^{-\frac{4}{5+dk^*}},
\end{align*}
where $k^*=\left(\frac{t}{\varepsilon}\right)^{\frac{\alpha}{\alpha-1}}3^{\frac{1}{\alpha-1}}$. So that \eqref{cor:MSE} gives
\begin{align*}
{\rm MSE}(t,z)=\underset{n\to+\infty}{O}\left(n^{-\frac{4}{5+d\left(\frac{t}{\varepsilon}\right)^{\frac{\alpha}{\alpha-1}}3^{\frac{1}{\alpha-1}}}}\right).
\end{align*}

Remark that if we formally take $\alpha=+\infty$, we get back to the situation of Example~\ref{example:1} and the upper bounds coincide as the previous upper bound writes
\begin{align*}
{\rm MSE}(t,z)=\underset{n\to+\infty}O\left(n^{-\frac{4}{5+d\frac{t}{\varepsilon}}}\right).
\end{align*}
Note that the rate of convergence of the mean squared error is in-between the traditional finite-dimensional rate \citep[see][]{Gyorfi2002}, and the rates obtained by Biau et al \cite{Biau2010} in the infinite-dimensional setting. This is explained by the particularity of our model which is itself in-between the finite and infinite-dimensional settings.
\end{example}

\section{Simulation study}\label{seq:Simu}

In this section we aim  at studying the performances of our estimator from a practical point of view. To this end we study our estimator over $n_{MC}$ replications of Monte Carlo simulations. The squared error, its mean (MSE) defined in \eqref{eq:MSE}, median, first and third empirical quartiles as well as the normalized root mean squared error (NRMSE) defined as follows
\begin{align}\label{eq:NRMSE}
\mathrm{NRMSE}\left(t,\vec Z_S(t)\right):=\frac{\sqrt{\mathrm{MSE}\left(t,\vec Z_S(t)\right)}}{\theta_S\left(t,\vec Z_S(t)\right)},
\end{align}
are used as indicators of the performances of our estimator and are calculated over a grid of $n_t$ times $t$ in $[0,1]$.

\subsection{Experimental design}

The Monte Carlo replications are simulated according to the model~\eqref{eq:model} presented in the introduction where $\Ss$ and $\theta_S$ are chosen as follows.
On the one hand the inter-arrival times of the counting process $M$ have the same distribution function than $U+\varepsilon$ with $U\sim\Ee(1/\varepsilon)$ for $\varepsilon>0$, putting ourselves in the situation of Example~\ref{example:1}. On the other hand
\begin{align}\label{eq:simulations}
\theta_S\left(t,\vec Z_S(t)\right) = \lambda_0(t)\exp\left(\sin\left(\beta\cdot\vec Z_S(t)\right)\right),
\end{align}
where $\lambda_0(t)=\frac{b}{a}\left(\frac{t}{a}\right)^{b-1}$ with $a,b>0$, $\beta$ is a vector of same dimension as $\vec Z_S(t)$ and $u\cdot v$ denotes the euclidian inner product of vectors $u$ and $v$. The co-process $Z$ is simulated according to a Brownian motion.

Note that the intensity $\theta$ is a modified version of the proportional hazards intensity function that models the dependence of our counting process on the past of the co-process~$Z$. For $\beta>0$, $\theta_S\left(t,\vec Z_S(t)\right)$ can be viewed as a stochastic perturbation of the intensity of a Weibull process as presented in Figure~\ref{fig:Regularity}.
Remark that for $\beta\ge 0.3$ the replications look too irregular for a kernel estimator to capture their behaviour properly. It is due to the increasing impact of the covariates which we chose to be Brownian.

%Expliquer simulation processus de Cox. Ce dont on a besoin, ce que ça implique.
As $(N_t)_{t\in[0,1]}$ is an inhomogeneous Poisson process conditionally on $\Zz$ and $\Ss$, we can simulate its jumping times by applying the inverse function of $\Lambda(\cdot)=\int_0^\cdot \theta_S\left(s,\vec Z_S(s)\right)\dd s$ to the jumping times of a homogeneous Poisson process with intensity $1$. In our case this inverse function writes
\begin{align*}
\Lambda^{-1}(u)=\Lambda_0^{-1}\left(\Lambda_0\left(S_{j_u}\right)+
\frac{u-\Lambda\left(S_{j_u}\right)}{\exp\left(\sin\left(\beta\cdot\vec Z_S(t)\right)\right)}\right),
\end{align*}
where $j_u$ is such that $\Lambda\left(S_{j_u}\right)\leq u<\Lambda\left(S_{j_u+1}\right)$ and $\Lambda_0(\cdot)=\int_0^\cdot \lambda_0(s)\dd s$. This allows us to simulate the data according to our model.

We finally take
\begin{align*}
\Hh(u) &=\left(\frac{1}{2}-\frac{5}{8}(3u^2-1)\right)\1_{\abs{u}\leq1},\\
K(u) &=(30u^2-36u+9)\1_{0\leq u\leq1},
\end{align*}
so that assumptions \ref{hyp:H} and \ref{hyp:K} are fulfilled.

\begin{figure}[ht]
\begin{center}
\includegraphics[scale=0.88]{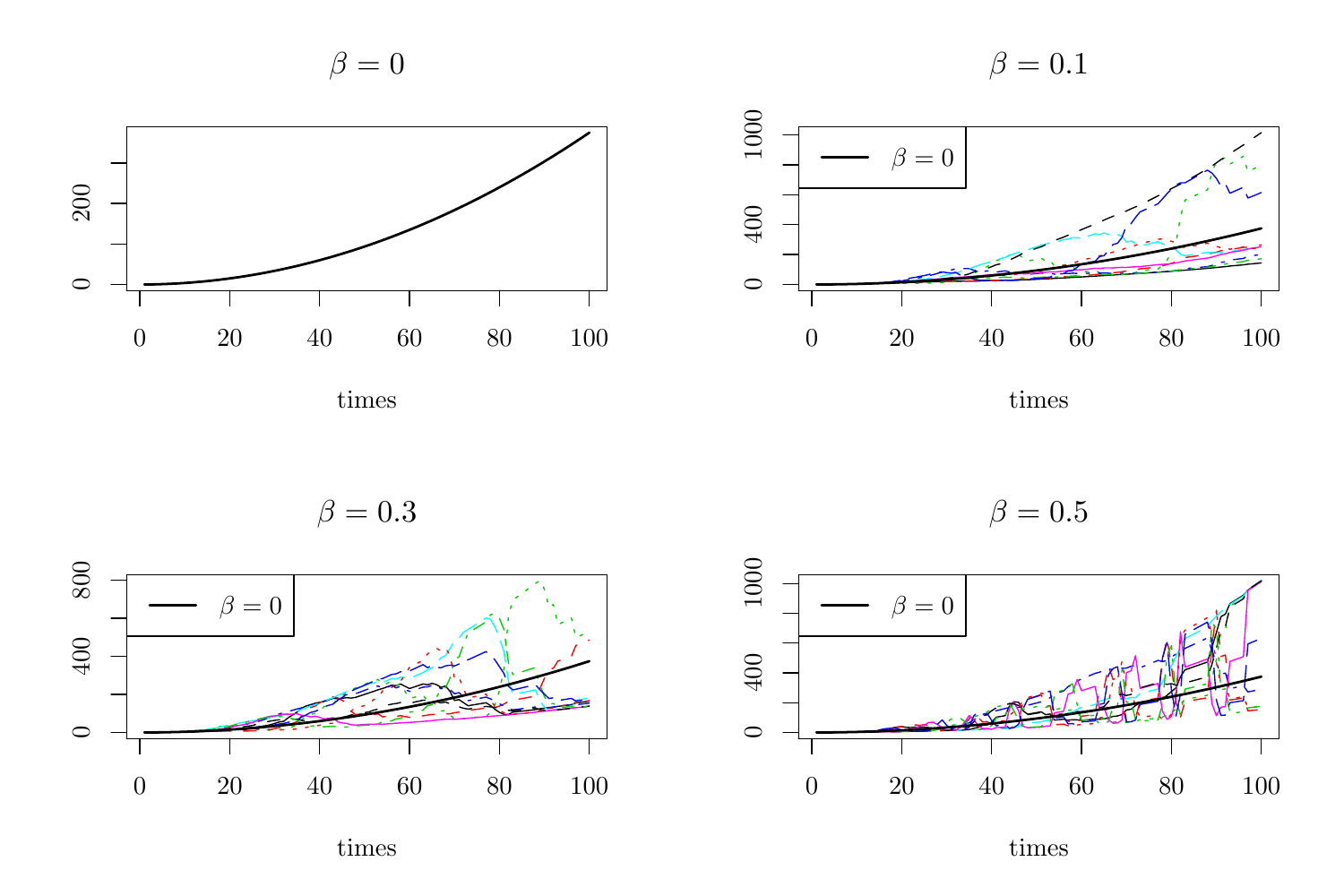}
\end{center}
\caption{\label{fig:Regularity} 10 replications of the studied intensity~\eqref{eq:simulations} for $\beta=0$ (top left), $\beta=0.1$ (top right), $\beta=0.3$ (bottom left) and $\beta=0.5$ (bottom right). The bold line represents the corresponding intensity for a Weibull process ($\beta=0$).}
\end{figure}

\subsection{Results}

%Courbes quantiles à commenter
Figure~\ref{fig:Quantiles} represents the theoretical intensity (solid) versus the first and third empirical quartiles (dashed and dotted) of $100$ Monte Carlo replications of our estimator for $n=500$ and $\beta=0$ (top left), $\beta=0.1$ (top right), $\beta=0.3$ (bottom left) and $\beta=0.5$ (bottom right). The case $\beta=0$ is the one of the estimation of the intensity of a Weibull process. As $\beta$ increases, the counting process $N$ deviates from this simple case to a point where the signal is almost chaotic due to the influence of the co-process $Z$ for $\beta=0.5$ (see Figure~\ref{fig:Regularity}). As expected our estimator is less accurate for high values of $t$ (\emph{i.e.} high dimensionality) and quickly varying objective function (\emph{e.g.} $\beta=0.5$). We also note an artifact for the estimation around zero. It is a well known issue with kernel estimation on the edges of the support of the objective function.

In the following we fix $\beta=0.1$.
\begin{figure}[ht]
\begin{center}
\includegraphics[scale=0.88]{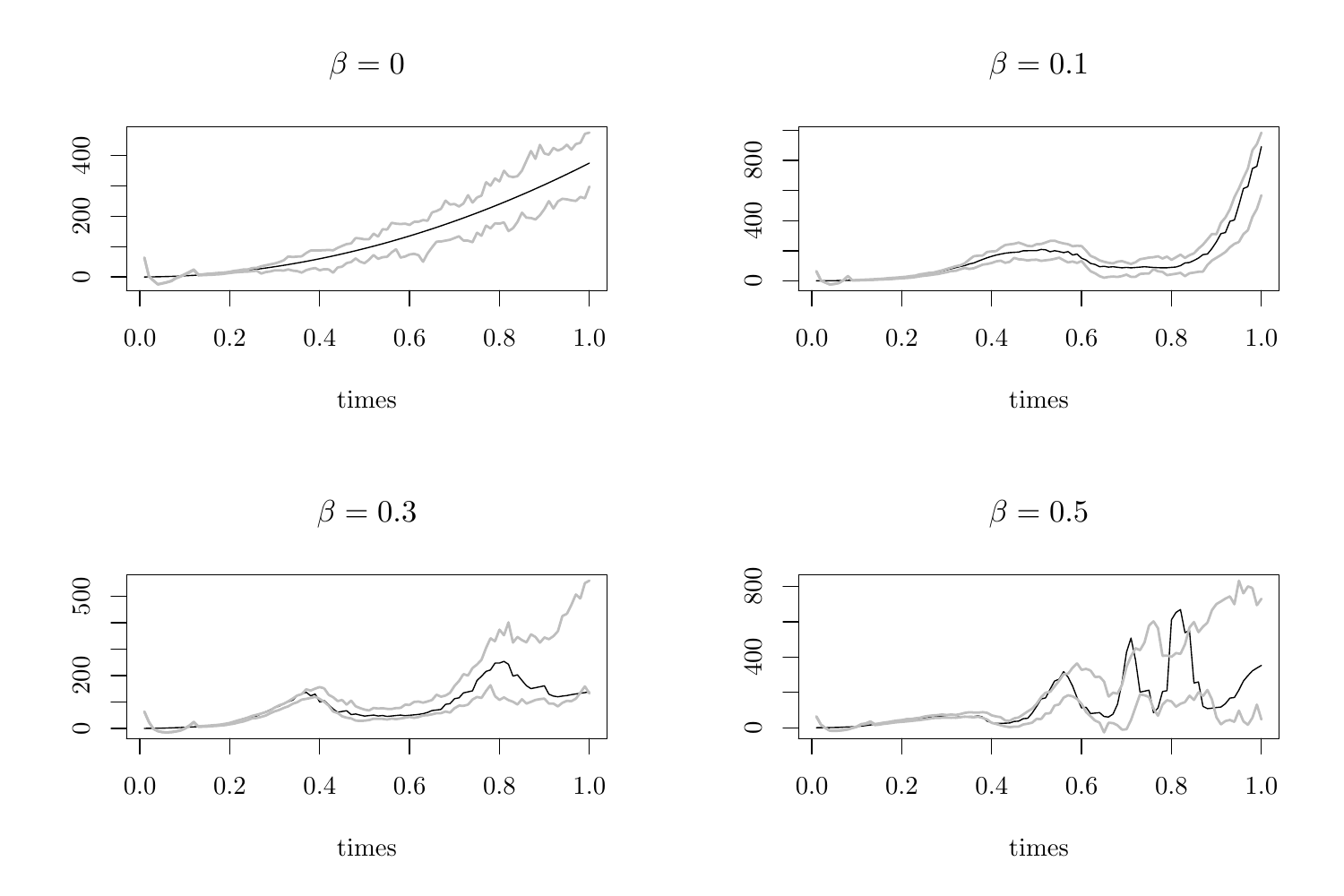}
\end{center}
\caption{\label{fig:Quantiles} Objective intensity (black line) versus first and third (gray lines) empirical quartiles for $500$ observations, $n_{MC}=100$, $n_t=100$ and $\beta=0$ (top left), $\beta=0.1$ (top right), $\beta=0.3$ (bottom left) and $\beta=0.5$ (bottom right).}
\end{figure}

Figure~\ref{fig:QuartSE} represents the median (solid line) and the first and third empirical quartiles (dashed and dotted lines) of the squared error of our estimator for $10,000$ Monte Carlo replications of our estimator for $n=500$ (Figure~\ref{subfig:n500QuartSE} and $n=10,000$ (Figure~\ref{subfig:n10000QuartSE}). As expected, the results are far better for $n=10,000$ where the third quartile does not exceed $14,000$ compared to a maximum of $50,000$ for $n=500$. Remark that these maxima are always attained near to $t=1$. That is explained by the fact that as $t$ increases, the dimension of the estimation problem increases.
\begin{figure}
\begin{subfigure}{\linewidth}
\begin{center}
\includegraphics[scale=0.8]{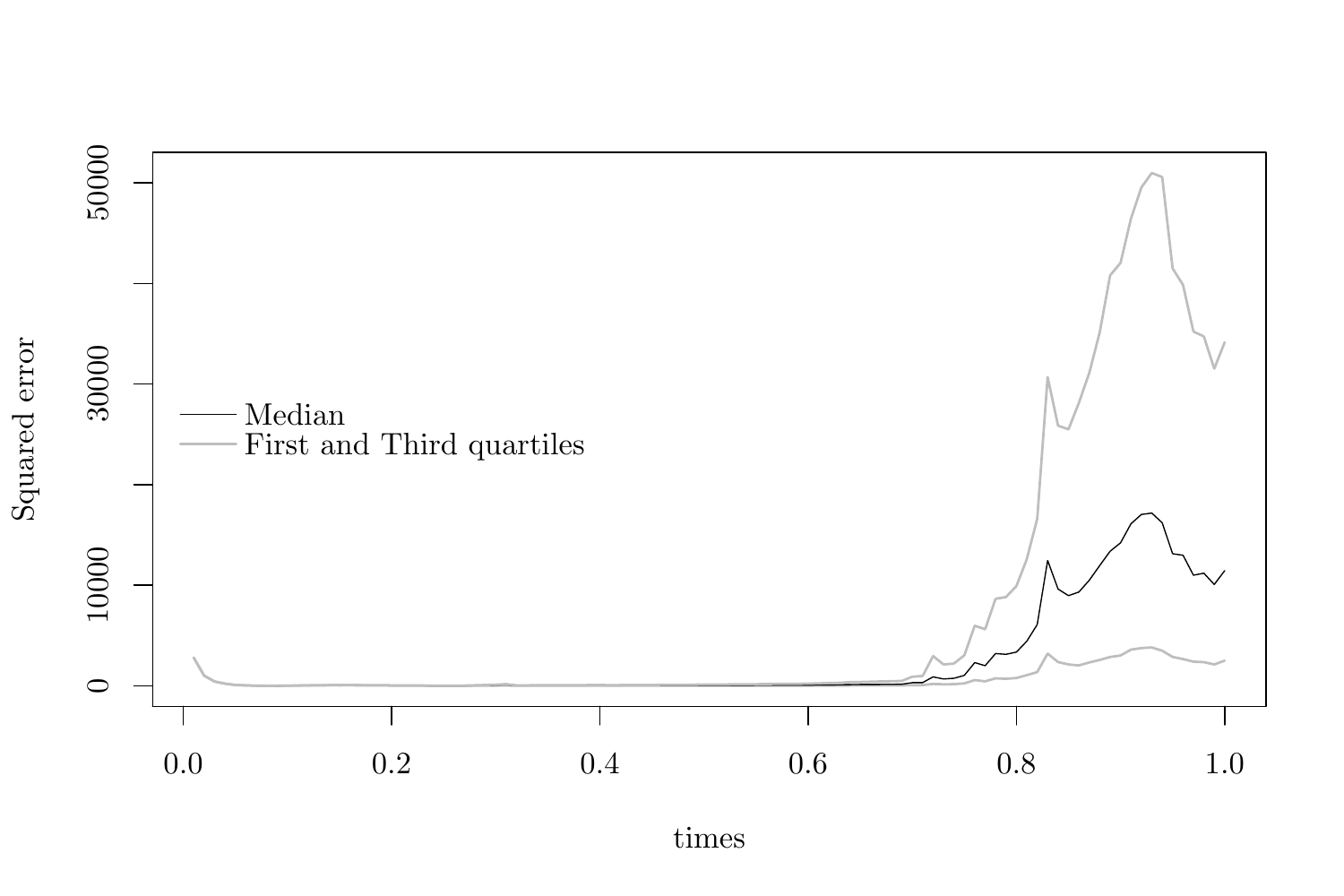}
\caption{$n=500$}
\label{subfig:n500QuartSE}
\end{center}
\end{subfigure}
\begin{subfigure}{\linewidth}
\begin{center}
\includegraphics[scale=0.8]{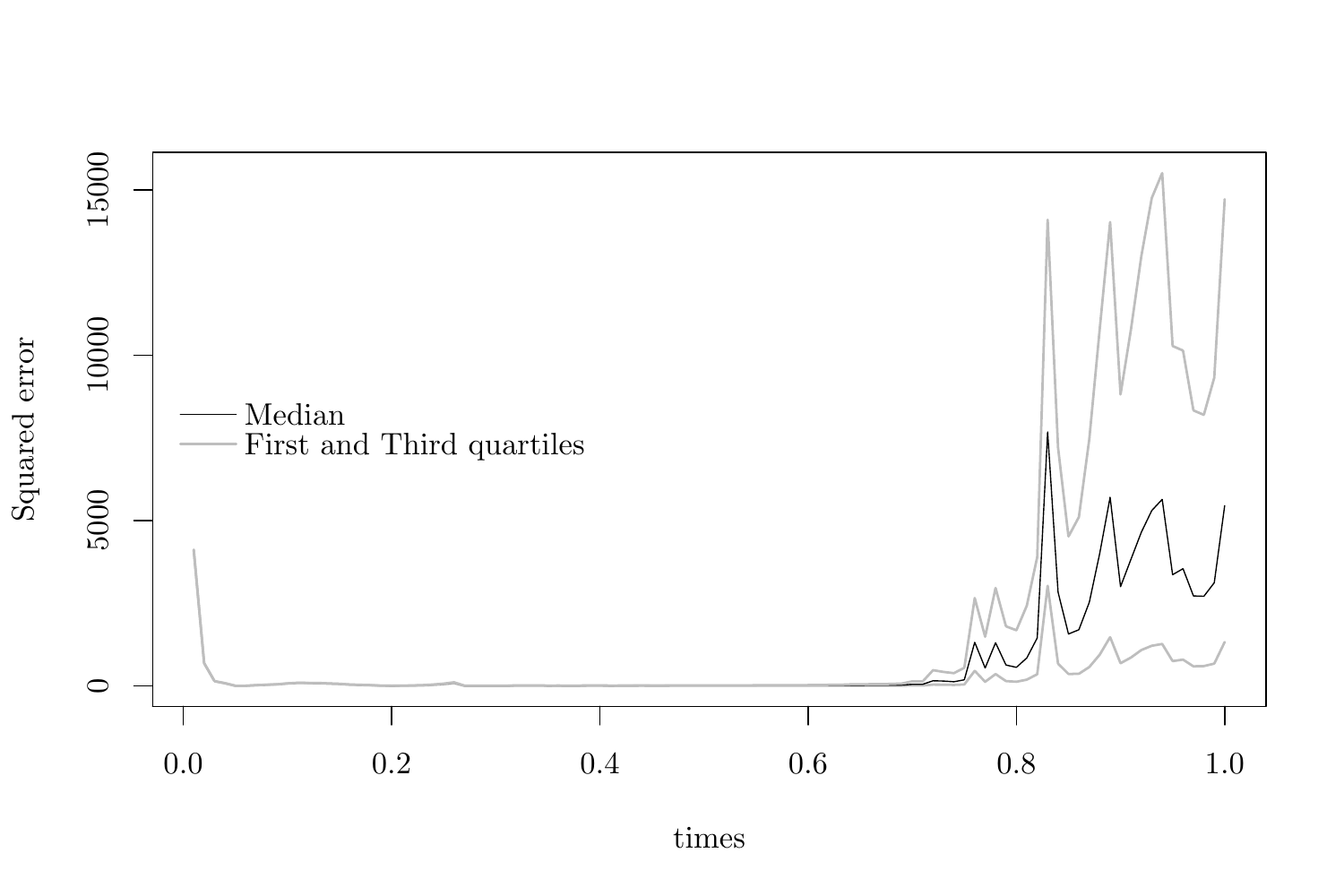}
\caption{$n=10,000$}
\label{subfig:n10000QuartSE}
\end{center}
\end{subfigure}
\caption{\label{fig:QuartSE} Median, first and third empirical quartiles of the squared error of our estimator for $n_{MC}=10,000$, $n_t=100$, $\beta=0.1$ and (a)~$n=500$ and (b)~$n=10,000$.}
\end{figure}

In Table~\ref{table:SE} the results are obtained for $10,000$ Monte Carlo replications of our model for $3$ different times and $5$ different values of $n$. The dimensionality of our estimation problem increases quickly towards $59$ at time $t=0.9$. This shows the difficulty of the estimation for small $n$. We observe nevertheless an increase in performance for bigger values of $n$. For $n=10,000$ the $\mathrm{NRMSE}$ indicator stays below $0.2$ after time $t=0.3$. At this point we seem to have attained the asymptotic property for the MSE described in Section~\ref{subsec:Results}. %More details are given in Table~\ref{table:SEAppendix} in the appendix.

\begin{table}
\begin{center}
\begin{tabular}{c|ccc|c|c|c}
  \hline
$n$ & $t$ & $M_t$ & $\lambda(t)$ & Estimate & MSE & NRMSE \\ 
  \hline
 & 0.5 &  33 & 36.21 & 38.15 & 4.7E+03 & 1.89 \\ 
100 & 0.7 &  47 & 110.17 & 97.50 & 1.2E+05 & 3.14 \\ 
 & 0.9 &  59 & 763.96 & 675.75 & 6.5E+06 & 3.34 \\ 
  \hline
 & 0.5 &  33 & 36.21 & 38.26 & 3.4E+02 & 0.51 \\ 
250 & 0.7 &  47 & 110.17 & 92.28 & 1.3E+03 & 0.33 \\ 
 & 0.9 &  59 & 763.96 & 717.24 & 3E+06 & 2.27 \\ 
  \hline
 & 0.5 &  33 & 36.21 & 38.37 &  96 & 0.27 \\ 
500 & 0.7 &  47 & 110.17 & 91.51 & 9.4E+02 & 0.28 \\ 
 & 0.9 &  59 & 763.96 & 649.71 & 4.8E+07 & 9.06 \\ 
  \hline
 & 0.5 &  33 & 36.21 & 38.03 &  54 & 0.20 \\ 
1,000 & 0.7 &  47 & 110.17 & 92.05 & 7.1E+02 & 0.24 \\ 
 & 0.9 &  59 & 763.96 & 732.71 & 7.2E+04 & 0.35 \\ 
  \hline
 & 0.5 &  33 & 36.21 & 37.84 &  11 & 0.09 \\ 
10,000 & 0.7 &  47 & 110.17 & 92.39 & 3.8E+02 & 0.18 \\ 
 & 0.9 &  59 & 763.96 & 743.70 & 2.9E+03 & 0.07 \\ 
\end{tabular}
\end{center}
\caption{Mean value of the estimator, mean squared error~\eqref{eq:MSE} and normalized root mean squared error~\eqref{eq:NRMSE} for $n_{MC}=10,000$ and $\beta=0.1$.}\label{table:SE}
\end{table}

\section{Application to real data}\label{seq:RealData}

We study a data set constituted of historical prices of $n=495$ companies as well as the crude oil prices over a period of roughly one year and two months (from 17th April, 2014 to 23rd June, 2015). The companies data are taken from the website Yahoo Finance so that every company considered composed the S\&P500 index on the 23rd June, 2015. The crude oil prices are taken from the website Investing.com.
The Cox process data consist of the count of the number of times when the percent returns of said companies go below a certain threshold with the counting rate depending on the stochastic dynamics of the company market capitalization. In our case the company market capitalization is represented by the action's trade volume normalized increments and is observed when the percent return of the crude oil action below another threshold. By analyzing this count, we aim to learn the financial properties of this $495$ companies system.

\begin{table}
\centering
\begin{tabular}{ccccccccc}
\hline
Date & Open & High & Low & Close & Volume & Adj.Close \\ 
  \hline
2015-06-23 & 39.89 & 39.95 & 39.42 & 39.60 & 2053600 & 39.60 \\ 
  2015-06-22 & 39.81 & 40.01 & 39.73 & 39.81 & 3901700 & 39.81 \\ 
  2015-06-19 & 39.80 & 39.94 & 39.49 & 39.49 & 2581000 & 39.49 \\ 
  2015-06-18 & 39.80 & 40.06 & 39.72 & 39.90 & 1865000 & 39.90 \\ 
  2015-06-17 & 39.76 & 39.80 & 39.32 & 39.60 & 1519400 & 39.60 \\ 
  2015-06-16 & 39.59 & 39.81 & 39.38 & 39.79 & 1422600 & 39.79 \\ 
  2015-06-15 & 39.63 & 39.63 & 39.25 & 39.52 & 2320100 & 39.52 \\ 
  2015-06-12 & 40.33 & 40.49 & 39.74 & 39.84 & 2764200 & 39.84 \\ 
  2015-06-11 & 40.57 & 40.60 & 40.29 & 40.53 & 1566000 & 40.53 \\ 
  2015-06-10 & 40.40 & 40.59 & 40.27 & 40.52 & 1787900 & 40.52 \\ 
  \multicolumn{7}{c}{$\ldots$}\\
\end{tabular}
\caption{{\label{table:RawData}}First $10$ rows of raw data for Agilent Technologies Inc. taken from Yahoo Finance.}
\end{table}

\begin{table}
\centering
\begin{tabular}{cccccccccc}
\hline
Date & Price & Open & High & Low & Vol. & Change \\ 
  \hline
2015-06-23 & 61.01 & 60.21 & 61.49 & 59.55 & 336.22K & 1.04\% \\ 
  2015-06-22 & 60.38 & 59.75 & 60.63 & 59.27 & 255.31K & 1.29\% \\ 
  2015-06-19 & 59.97 & 60.88 & 60.93 & 59.24 & 299.89K & -1.40\% \\ 
  2015-06-18 & 60.82 & 60.10 & 61.33 & 59.67 & 171.48K & 0.81\% \\ 
  2015-06-17 & 60.33 & 60.52 & 61.81 & 59.34 & 232.09K & -0.20\% \\ 
  2015-06-16 & 60.45 & 60.01 & 60.81 & 59.88 & 129.30K & 0.75\% \\ 
  2015-06-15 & 60.00 & 60.33 & 60.42 & 59.19 & 128.26K & -0.66\% \\ 
  2015-06-12 & 60.40 & 60.92 & 61.06 & 60.18 & 91.96K & -1.34\% \\ 
  2015-06-11 & 61.22 & 61.56 & 61.91 & 60.65 & 150.62K & -0.97\% \\ 
  2015-06-10 & 61.82 & 61.00 & 62.22 & 60.88 & 188.78K & 2.00\% \\ 
  \multicolumn{7}{c}{$\ldots$}\\
\end{tabular}
\caption{{\label{table:RawOil}}First $10$ rows of raw data for the crude oil action taken from Investing.com.}
\end{table}

%\subsection{Preprocessing of the data}

Tables~\ref{table:RawData} and \ref{table:RawOil} present the layout of the raw data directly taken from the websites Yahoo Finance and Investing.com.
We denote $\left((Y^1_t)_{t\in[0,1]},\ldots,(Y^n_t)_{t\in[0,1]}\right)$ the \texttt{Open} columns of the Yahoo Finance raw data. It represents the daily open prices of the actions of the companies. We define the percent returns as follows $X_t^k:=\frac{Y_t^k-Y_{t-1}^k}{Y_{t-1}^k}$ for $1\leq k\leq n$. Denote $(\zeta_t)_{t\in[0,1]}$ the \texttt{Open} column of the Investing.com raw data. In the same way as for the Yahoo Finance raw data it represents the daily open prices of the crude oil action. We define its percent returns by $\xi_t:=\frac{\zeta_t-\zeta_{t-1}}{\zeta_{t-1}}$. The \texttt{Volume} columns of the Yahoo Finance data are denoted $\left((W_t^1)_{t\in[0,1]},\ldots,(W_t^n)_{t\in[0,1]}\right)$. Its normalized increments are defined for $1\leq k\leq n$ by $Z_t^k:=\frac{W_t^k-W_{t-1}^k}{W_{t-1}^k}$.

Random times $S$ as well as the counting process $N$ are deduced from these transformed data sets. They are defined such that $S_1$ is the first time at which $\xi$ goes below $\alpha:=-0.01$, $S_2$ is the second time, \emph{etc.} and the process $N$ counts the number of times $X$ goes below $\beta:=-0.015$. Figure~\ref{fig:CountingProcessCreation} represents the trajectory of $\xi$ and gives an illustration of the construction of $S$. With this construction we get a total of $50$ times of observation. Remark that these thresholds represent a $1\%$ drop for the percent returns of the crude oil action and a $1.5\%$ drop for the percent returns of the $\mathrm{S}\& \mathrm{P}500$ companies action.

\begin{figure}[H]
\begin{center}
\includegraphics[scale=0.554]{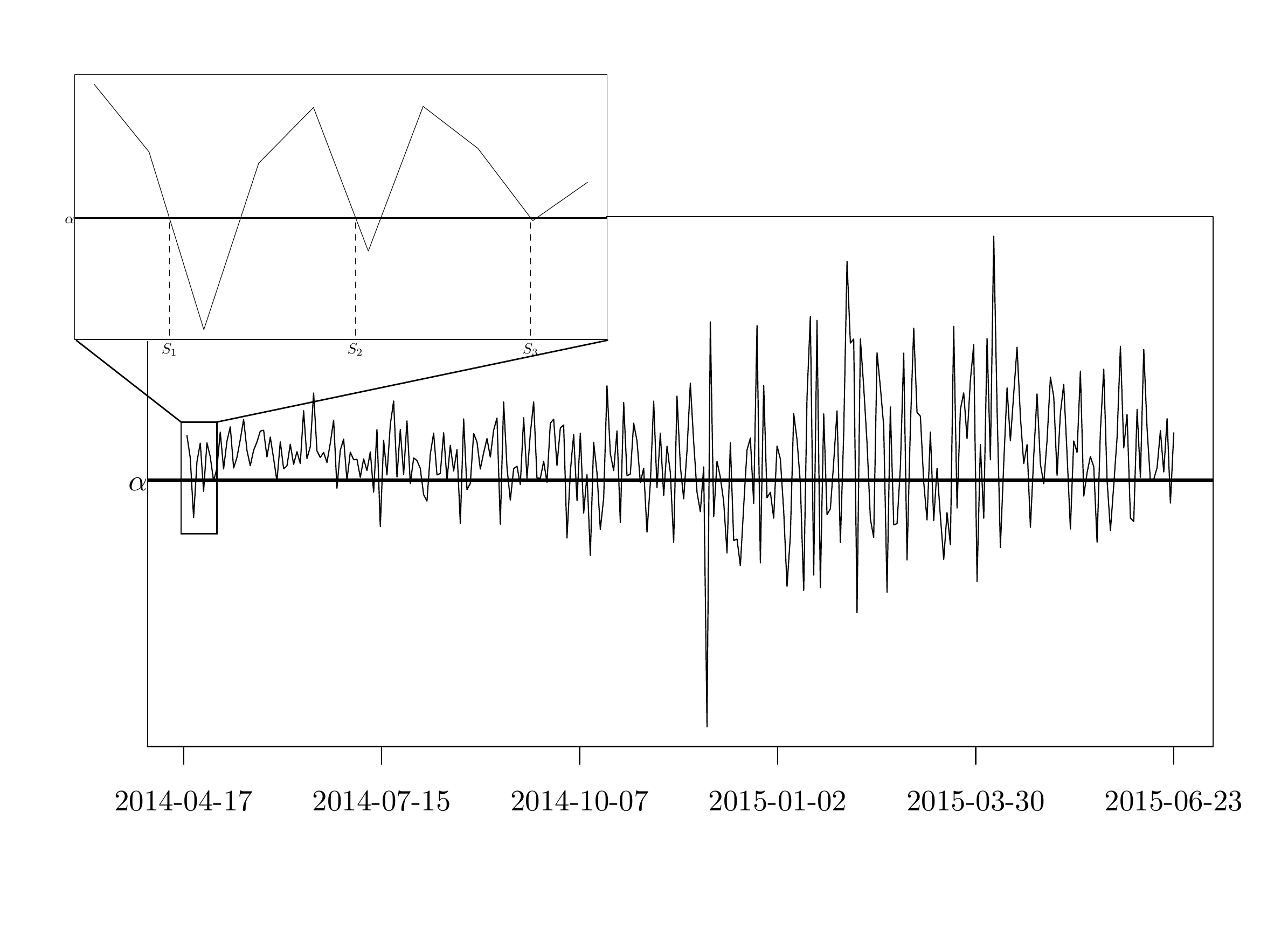}
\end{center}
\caption{\label{fig:CountingProcessCreation} Crude oil action percent returns plus a zoom on a small window of time to demonstrate the construction of the random times $S$.}
\end{figure}

We aim to compare the inhomogeneous Poisson model with our model~\eqref{eq:model}. To this end we compute our estimator over the time span defined by the data and for $10$ chosen trajectories of the covariate process $Z$. The resulting estimated intensities are given in Figure~\ref{fig:Nclust10}. In most cases ($7$ out of $10$), we estimate the same intensity as in the inhomogeneous Poisson model. In the second graph we observe that for $3$ trajectories of $Z$, taking covariates into consideration in the model provides estimations that stays close to the inhomogeneous Poisson model at first and deviates from it after a short moment.

\begin{figure}[H]
\begin{subfigure}{\linewidth}
\begin{center}
\includegraphics[scale=0.8]{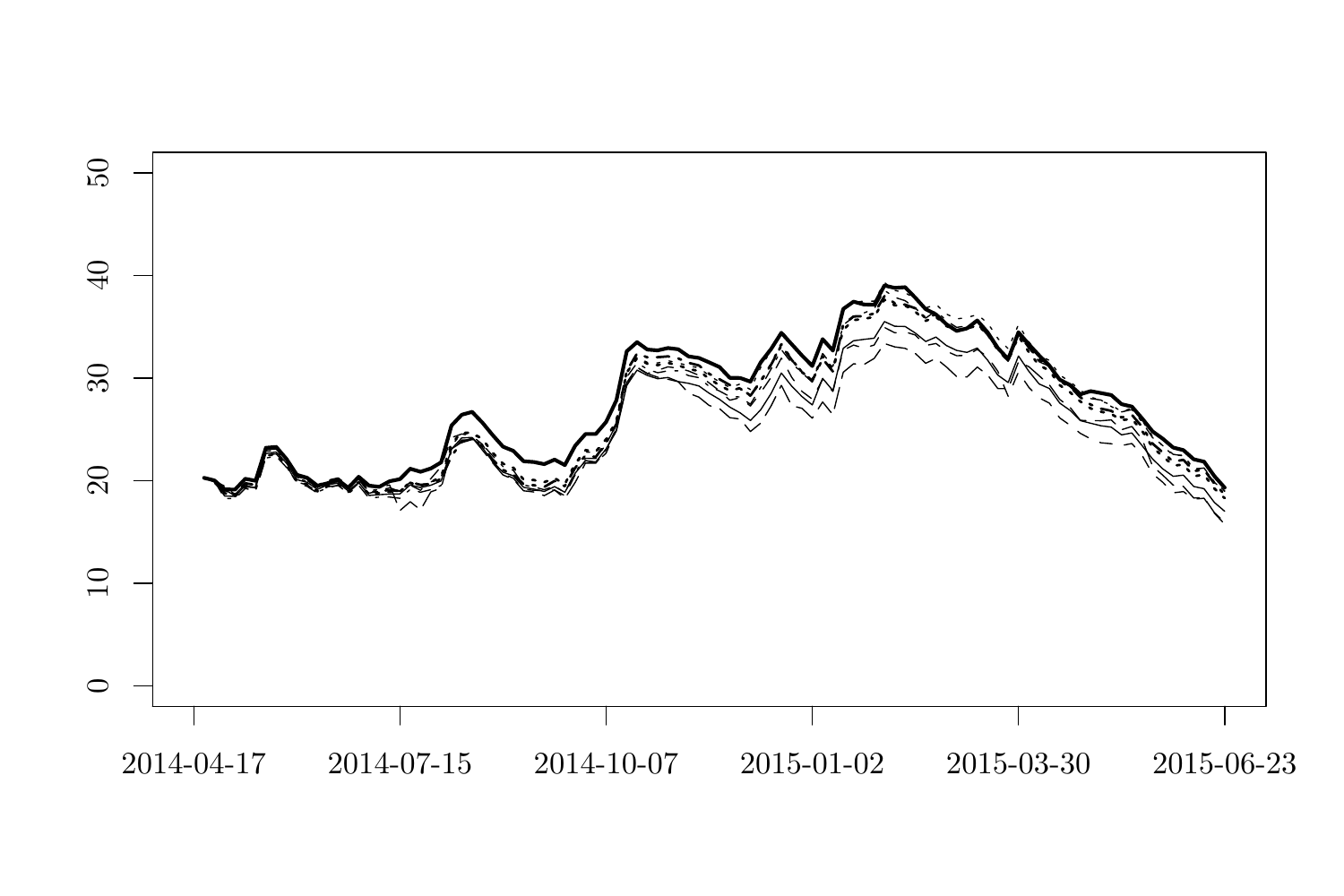}
\end{center}
\end{subfigure}
\begin{subfigure}{\linewidth}
\begin{center}
\includegraphics[scale=0.8]{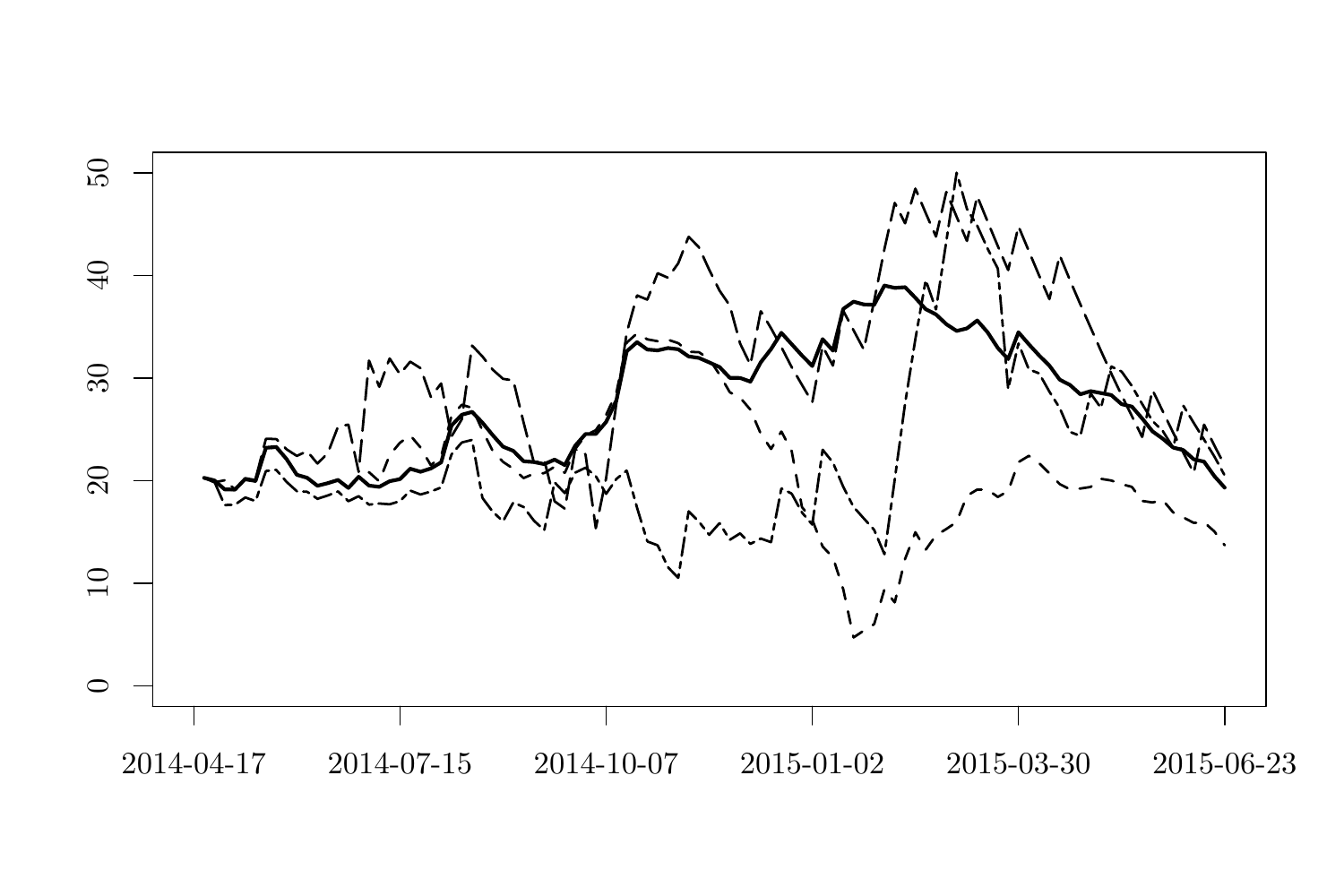}
\end{center}
\end{subfigure}
\caption{\label{fig:Nclust10} Estimation of the intensity function $\lambda$ in the Cox process model~\eqref{eq:model} for $10$ chosen trajectories of the covariate process $Z$ compared to the estimation for an inhomogeneous Poisson model (bold solid line)}
\end{figure}

\section{Proofs}\label{seq:Proofs}

In the sequel, $C$ denotes a positive constant that can change of values from line to line, $\PP_\Ss$, $\EE_\Ss$ and $\tilde\EE$ respectively stand for $\PP(\cdot|\Ss)$, $\EE(\cdot|\Ss)$ and $\EE(\cdot|\Zz,\Ss)$. For fixed $t\in[0,1]$ and $z:[0,1]\to\RR$, we define
\[\tilde \phi_{S,h,\eta}\bigl(t,\vec z_S(t)\bigr)=\frac{1}{n}\sum_{k=1}^n H_\eta\left(\vec z_S(t)-\vec Z_S^k(t)\right)\int_0^t K_h(t-s)\theta_S\left(s,\vec Z_S^k(s)\right)\dd s,\]
and $\phi\bigl(t,\vec z_S(t)\bigr)=f\bigl(\vec z_S(t)\bigr)
\theta_S\bigl(t,\vec z_S(t)\bigr)$.

\subsection{Proof of Theorem~\ref{theo:MSEcond}}

\begin{proof}

Define for $k=1,\ldots,n$ the trimmed version of the Nadaraya-Watson weights as
\begin{align*}
\Ww_{S,\eta}^k\bigl(t,\vec z_S(t)\bigr)=\frac{\frac{1}{n}H_\eta\left(\vec z_S(t)-\vec Z_S^k(t)\right)}{\frac{1}{n}\sum_{l=1}^n H_\eta\left(\vec z_S(t)-\vec Z_S^l(t)\right)\vee a_n},
\end{align*}
and consider
\begin{align*}
\bar N^k_t= N^k_t-\Lambda^k_t,\quad t\in[0,1],
\end{align*}
with $\displaystyle\Lambda^k_t=\tilde \EE N_t^k=\int_0^t \theta_S\left(s,\vec Z_S^k(s)\right)\dd s$.

Using these notations we have
\begin{align}
\tilde\theta_{S,h,\eta}\bigl(t,\vec z_S(t)\bigr)-\theta_S\bigl(t,\vec z_S(t)\bigr)&=A+B,\label{eq:decomposition}
\end{align}
where
\begin{align*}
A&=\sum_{k=1}^n\Ww_{S,\eta}^k\bigl(t,\vec z_S(t)\bigr)\int_0^t K_h(t-s)\dd \bar N_s^k, \text{ and}\\
B&=\sum_{k=1}^n\Ww_{S,\eta}^k\bigl(t,\vec z_S(t)\bigr)\int_0^t K_h(t-s)\theta_S\left(s,\vec Z_S^k(s)\right)\dd s-\theta_S\bigl(t,\vec z_S(t)\bigr).
\end{align*}
Decomposition~\eqref{eq:decomposition} gives
\begin{align}
{\rm MSE}_{\Ss}(t,z)&\leq 2\left(\EE_\Ss A^2+\EE_\Ss B^2\right).\label{eq:A2+B2}
\end{align}
On the one hand, since $\left((\bar N^k_t)_{t\in[0,1]}:k=1,\ldots,n\right)$ are independent trajectories of a conditional local martingale given $\Ss$ and $\Zz$, 
\begin{align*}
\EE_\Ss A^2&=\EE_\Ss \left(\sum_{1\leq k\leq n}\left[\Ww_{S,\eta}^k\bigl(t,\vec z_S(t)\bigr)\right]^2\tilde\EE\left[\left(\int_0^t K_h(t-s)\dd \bar N_s^k\right)^2\right]\right)\\
&\leq \EE_\Ss \left(\sup_{k=1,\ldots,n}\Ww_{S,\eta}^k\bigl(t,\vec z_S(t)\bigr)\tilde\EE\left[\left(\int_0^t K_h(t-s)\dd \bar N_s^1\right)^2\right]\right).
\end{align*}
%and 
%\begin{align*}
%B&=\EE_\Ss\left(\left[\frac{\tilde \phi_{S,h,\eta}\bigl(t,\vec z_S(t)\bigr)}{\tilde f_{S,\eta}\bigl(\vec z_S(t)\bigr)}-\frac{\phi\bigl(t,\vec z_S(t)\bigr)}{f\bigl(\vec z_S(t)\bigr)}\right]^2\right).
%\end{align*}

As $(N^1_t)_{t\in[0,1]}$  is a conditional Poisson process with cumulative intensity function $\Lambda_t^1$ given $\Ss$ and $\Zz$, the quadratic variation of $(\bar N_t^1)_{t\in[0,1]}$ is the process $(N^1_t)_{t\in[0,1]}$. The Itô isometry hence gives
\begin{align*}
\tilde \EE \left[\left(\int_0^t K_h(t-s)\dd \bar N_s^1\right)^2\right] = \int_0^t K_h^2(t-s)\theta_S\left(t,\vec Z_S^1(s)\right)\dd s.
\end{align*}
So that 
\begin{align}
\EE_\Ss A^2&\leq \EE_\Ss \left[\sup_{k=1,\ldots,n}\Ww_{S,\eta}^k\bigl(t,\vec z_S(t)\bigr)\int_0^t K_h^2(t-s)\theta_S\left(s,\vec Z_S^1(s)\right)\dd s\right] \nonumber\\
&\leq \frac{ \norm{K}_2^2\Theta^{M_t}\norm{\Hh}_\infty^{dM_t} }{ nh\eta^{dM_t} } \EE_\Ss\left[\frac{1}{\tilde f_{S,\eta}\bigl(\vec z_S(t)\bigr)}\right] \nonumber\\
&\leq \frac{ \norm{K}_2^2\Theta^{M_t}\norm{\Hh}_\infty^{dM_t} }{ nh\eta^{dM_t} } \left(\frac{1}{F_0^{M_t}}+\EE_\Ss\abs{\frac{1}{\tilde f_{S,\eta}(\vec z_S(t)}-\frac{1}{f(\vec z_S(t)}}\right).\label{eq:A}
\end{align}
Combining \eqref{eq:A2+B2}, \eqref{eq:A}, lemma~\ref{lemma:I} and lemma~\ref{lemma:II}, theorem follows.

\end{proof}

\subsection{Proof of Proposition~\ref{prop:CvProba}}

\begin{proof}
Let $\varepsilon>0$, and define
\begin{align*}
\Aa&=\left\{\abs{\tilde \theta_{S,h,\eta}\big(t,\vec z_S(t)\big)-\theta_S\big(t,\vec z_S(t)\big)}>\varepsilon\right\},\text{ and}\\
\Bb&=\left\{\abs{\tilde f_{S,\eta}\big(\vec z_S(t)\big)-f\big(\vec z_S(t)\big)}\leq\frac{F_0}{2}\right\}.
\end{align*}
Then
\begin{align*}
\PP_\Ss(\Aa)\le \PP_\Ss(\Aa\cap\Bb)+\PP_\Ss(\bar\Bb).
\end{align*}
On the one hand, 
\begin{align*}
\EE_\Ss\abs{\tilde f_{S,\eta}\big(\vec z_S(t)\big)-f\big(\vec z_S(t)\big)}^2\leq\EE_\Ss\abs{\hat f_{S,\eta}\big(\vec z_S(t)\big)-f\big(\vec z_S(t)\big)}^2+a_n^2,
\end{align*}
usual properties on kernel estimation of the density \citep[see][]{Bosq1987} gives 
\[\tilde f_{S,\eta}\big(\vec z_S(t)\big)\xrightarrow[]{\LL^2}f\big(\vec z_S(t)\big)\text{ under }\PP_\Ss\]
 so that $\PP_\Ss(\bar\Bb)\xrightarrow[n\to+\infty]{}0$.

On the other hand,
\begin{align*}
\Aa\cap\Bb=\left\{\abs{\hat\phi_{S,h,\eta}\big(t,\vec z_S(t)\big) f\big(\vec z_S(t)\big)-\phi\big(t,\vec z_S(t)\big)\tilde f_{S,\eta}\big(\vec z_S(t)\big)}>\varepsilon\frac{F_0^2}{2}\right\},
\end{align*}
and Lemma \ref{lemma:phi} give $\PP_\Ss(\Aa\cap\Bb)\xrightarrow[n\to+\infty]{}0$. Combining these results with the dominated convergence theorem, proposition follows.

\end{proof}

\subsection{Proof of Theorem~\ref{theo:CvDistr}}

\begin{proof}
Define 
\begin{align*}
A_n&=\alpha_n(nh\eta^{dM_t})^{1/2} \left[\tilde f_{S,\eta} \bigl(\vec z_S(t)\bigr)-f\bigl(\vec z_S(t)\bigr)\right], \text{ and}\\
B_n&=\frac{\tilde f_{S,\eta} \bigl(\vec z_S(t)\bigr)} {f \bigl(\vec z_S(t)\bigr)} (nh\eta^{dM_t})^{1/2} \frac{\hat\phi_{S,h,\eta} \bigl(t,\vec z_S(t)\bigr)-\phi \bigl(t,\vec z_S(t)\bigr)} {\left[\hat \phi_{S,h,\eta} \bigl(t,\vec z_S(t)\bigr) \norm{K}^2_2\norm{\Hh}_2^{2dM_t}\right]^{1/2}},
\end{align*}
where 
\begin{align*}
\alpha_n=\frac{\sqrt{\hat \phi_{S,h,\eta} \bigl(t,\vec z_S(t)\bigr) }}{f \bigl(\vec z_S(t)\bigr)\norm{K}_2 \norm{\Hh}_2^{dM_t}}.
\end{align*}
On the one hand,
\begin{align*}
\alpha_n \xrightarrow[]{\PP_\Ss} \sqrt{\frac{\theta_S\big(t,\vec z_S(t)\big)}{f\big(\vec z_S(t)\big)\norm{K}^2_2\norm{\Hh}^{2dM_t}_2}}\leq \sqrt{\frac{\norm{\theta_S}_\infty}{F_0\norm{K}^2_2\norm{\Hh}^{2dM_t}_2}},
\end{align*}
and
\begin{align*}
\EE_\Ss\abs{\tilde f_{S,\eta}\bigl(\vec z_S(t)\bigr)-f\bigl(\vec z_S(t)\bigr)}^2\leq\EE_\Ss\abs{\hat f_{S,\eta}\bigl(\vec z_S(t)\bigr)-f\bigl(\vec z_S(t)\bigr)}^2+a_n^2,
\end{align*}
usual properties on kernel estimation of the density \citep[see][]{Bosq1987} gives
\begin{align*}
nh\eta^{dM_t}\EE_\Ss\abs{\tilde f_{S,\eta}\bigl(\vec z_S(t)\bigr)-f\bigl(\vec z_S(t)\bigr)}^2\xrightarrow[n\to+\infty]{}0.
\end{align*}

Combining these results with Slutsky lemma gives
\begin{align}
A_n\xrightarrow[n\to+\infty]{\PP_\Ss} 0. \label{eq:An}
\end{align}

On the other hand,
\begin{align*}
\EE_\Ss\hat\phi_{S,h,\eta}\big(t,\vec z_S(t)\big)=(\Kk_{h,\eta}*\psi)\big(t,\vec z_S(t)\big),
\end{align*}
with $\Kk$ the product kernel of $H$ and $K$ and \[\forall(u,y)\in[0,t]\times\RR^{dM_t},\psi(u,y)=f(y)\theta_S\big(u,\vec y_S(u)\big).\] Applying Lemma~\ref{lemma:WandJones} with kernel $\mathcal{K}$ and function $\psi$ gives
\begin{align*}
\EE_\Ss\hat\phi_{S,h,\eta}\big(t,\vec z_S(t)\big)\leq\phi\big(t,\vec z_S(t)\big)+\frac{(dM_t+1)^2}{2}\max(h^2,\eta^2)Q_{M_t}C_{\Hh,K},
\end{align*}
where $C_{\Hh,K}=\max\left(\int_\RR z^2\Hh(z)\dd z,\int_\RR u^2K(u)\dd u\right)$ is a finite constant since $\Hh$ and $K$ are compactly supported.
As $nh^5\eta^{dM_t}\xrightarrow[n\to+\infty]{} 0$ and $nh\eta^{dM_t+4}\xrightarrow[n\to+\infty]{} 0$,
\begin{align*}
(nh\eta^{dM_t})^{1/2}\frac{\EE_\Ss\hat\phi_{S,h,\eta} \bigl(t,\vec z_S(t)\bigr)-\phi \bigl(t,\vec z_S(t)\bigr)} {\left[ \phi \bigl(t,\vec z_S(t)\bigr) \norm{K}^2_2\norm{\Hh}_2^{2dM_t}\right]^{1/2}} \xrightarrow[n\to+\infty]{}0.
\end{align*}
Combining this result with Lemma~\ref{lemma:phiDistr} and Slutsky lemma gives
\begin{align}
B_n\toDd \Nn(0,1),\text{ under }\PP_\Ss.\label{eq:Bn}
\end{align}

As
\begin{align*}
(nh\eta^{dM_t})^{1/2}\frac{\tilde \theta_{S,h,\eta}\bigl(t,\vec z_S(t)\bigr)-\theta_S\bigl(t,\vec z_S(t)\bigr)}{\left[\hat \phi_{S,h,\eta} \bigl(t,\vec z_S(t)\bigr)\norm{K}^2_2\norm{\Hh}_2^{2dM_t}/\tilde f_{S,\eta}^2\bigl(\vec z_S(t)\bigr)\right]^{1/2}}=B_n-A_n ,
\end{align*}
using \eqref{eq:An} and \eqref{eq:Bn}, we have
\begin{align*}
\EE_\Ss\left(e^{iu(B_n-A_n)}\right)\xrightarrow[n\to+\infty]{} e^{-\frac{u^2}{2}}.
\end{align*} 
Combining this result with the dominated convergence theorem, theorem follows.

\end{proof}

\section{Appendix}\label{seq:Appendix}

In the sequel $C$ denotes a positive constant under $\PP_\Ss$ that can change of values from line to line and $\var_\Ss$ stands for $\var(\cdot|\Ss)$. For simplicity, we may use the notations $\tilde f$, $\tilde \theta$, $\hat \phi$ and $\tilde \phi$ instead of $\tilde f_{S,\eta}$, $\tilde \theta_{S,h,\eta}$, $\hat \phi_{S,h,\eta}$ and $\tilde \phi_{S,h,\eta}$ respectively.

\begin{lemma}\label{lemma:WandJones}
Let $\mathbf{K}$ be a bounded, compactly supported $d$-variate kernel satisfying
\begin{align*}
\int_{\RR^d} \mathbf{K}(z)\dd z=1\text{ and }\int_{\RR^d} z\mathbf{K}(z)\dd z=0.
\end{align*}

Let $\mathbf{h}=\mathrm{diag}(h_1^2,\ldots,h_d^2)$ be such that $n^{-1}\abs{\mathbf{h}}^{-1/2}$ and all entries of $\mathbf{h}$ approach zero as $n$ tends to $+\infty$.  Also, we assume that the ratio of the largest and smallest eigenvalues of $\mathbf{h}$ is bounded for all $n$.

Denote $\mathbf{K}_\mathbf{h}(z)=\abs{\mathbf{h}}^{-1/2} \mathbf{K}(\mathbf{h}^{-1/2}z).$

Let $f$ be a $d$-variate function. Also, let $\mathfrak{D}_f(z)$ be the vector of first-order partial derivatives of $f$ and $\mathfrak{H}_f(z)$ be the Hessian matrix of $f$. Let's assume that for all $1\leq k\leq d$, \[\forall y\in\RR^d, \lnorm{\frac{\partial^2 f}{\partial x_k^2}(y)}_ 2\leq F_d,\] where $F_d$ is a positive constant depending on dimension $d$.

Then 
\begin{align*}
\abs{\int_{\RR^d} \mathbf{K}_{\mathbf{h}}(x-y)f(y)\dd y- f(x)}\leq \frac{d}{2}\norm{\mathbf{h}}_2 F_d\int_{\RR^d}\norm{z}_2^2\mathbf{K}(z)\dd z.
\end{align*}
\end{lemma}

\begin{proof}
By the multivariate version of Taylor's theorem with Lagrange remainder, for some $\gamma\in(0,1)$,
\begin{align*}
\abs{\int_{\RR^d} \mathbf{K}_{\mathbf{h}}(x-y)f(y)\dd y-f(x)}&\leq\abs{\frac{1}{2}\int_{\RR^d} z^\top \mathbf{h}^{1/2}\mathfrak{H}_f(x-\gamma\mathbf{h}^{1/2}z) \mathbf{h}^{1/2}z\mathbf{K}(z)\dd z }\\
&\leq \frac{1}{2} \norm{\mathbf{h}}_2 \sup_{y\in\RR^d} \norm{\mathfrak{H}_f(y)}_2 \int_{\RR^d} \norm{z}^2_2 \mathbf{K}(z)\dd z \\
&\leq \frac{d}{2}\norm{\mathbf{h}}_2 F_d \int_{\RR^d} \norm{z}_2^2 \mathbf{K}(z)\dd z,
\end{align*}
as $\displaystyle\int_{\RR^d} \mathbf{K}(z)\dd z=1,$ and $\displaystyle\int_{\RR^d} z\mathbf{K}(z)\dd z=0$.
\end{proof}

\begin{lemma}\label{lemma:I}
Under the assumptions of Theorem~\ref{theo:MSEcond}, for $p= 1$
\begin{align*}
\EE_\Ss \abs{\frac{1}{\tilde f_{S,\eta}\big(\vec z_S(t)\big)}-\frac{1}{f\big(\vec z_S(t)\big)}}^p &\!\!\! \leq \Clemma{0}\bigg((n\eta^{dM_t})^{\varepsilon-1}  + \eta^{2} + \frac{1}{(n\eta^{dM_t})^{1/2}}\\
&\qquad + \frac{n\eta^{dM_t+4}}{(n\eta^{dM_t})^{\varepsilon}} + \frac{1}{(n\eta^{dM_t})^{\varepsilon}}\bigg),
\end{align*}
for $p=2$,
\begin{align*}
\EE_\Ss \abs{\frac{1}{\tilde f_{S,\eta}\big(\vec z_S(t)\big)}-\frac{1}{f\big(\vec z_S(t)\big)}}^p & \leq \Clemma{1} \bigg((n\eta^{dM_t})^{2\varepsilon-2} + \frac{1}{n\eta^{dM_t}} + \eta^{4} +\frac{1}{(n\eta^{dM_t})^{2\varepsilon}}\\
&\qquad + \frac{1}{(n\eta^{dMt})^{2\varepsilon+1}} + \frac{(n\eta^{dMt+4})^2}{(n\eta^{dM_t})^{2\varepsilon}}\bigg),
\end{align*}
and for $p>2$
\begin{align*}
\EE_\Ss \abs{\frac{1}{\tilde f_{S,\eta}\big(\vec z_S(t)\big)}-\frac{1}{f\big(\vec z_S(t)\big)}}^p & \leq \Clemma{2} \Bigg((n\eta^{dM_t})^{p\varepsilon-p} + \frac{1}{(n\eta^{dM_t})^{p/2}} + \frac{1}{(n\eta^{dM_t})^{p-1}}\\
&\qquad + \eta^{2p} + \frac{1}{(n\eta^{dM_t})^{p\varepsilon}} + \frac{1}{(n\eta^{dMt})^{p\varepsilon+p-1}}\\
&\qquad  + \frac{(n\eta^{dM_t+4})^p}{(n\eta^{dM_t})^{p\varepsilon}}\Bigg),
\end{align*}

where $\Clemma{0}$, $\Clemma{1}$ and $\Clemma{2}$ depend only on $d$, $M_t$, $Q_{M_t}$, $F_0,$ $F_\infty$, $\norm{\Hh}_2$, $\norm{\Hh}_\infty$ and $\int_{\RR}z^2\Hh(z)\dd z$. 
\end{lemma}

\begin{proof}
Since
\begin{align*}
\frac{1}{\tilde f\big(\vec z_S(t)\big)}=\frac{1}{f\big(\vec z_S(t)\big)}-\frac{\tilde f\big(\vec z_S(t)\big)-f\big(\vec z_S(t)\big)}{f^2\big(\vec z_S(t)\big)}+\frac{\Big(\tilde f\big(\vec z_S(t)\big)-f\big(\vec z_S(t)\big)\!\Big)^2}{f^2\big(\vec z_S(t)\big)\tilde f\big(\vec z_S(t)\big)},
\end{align*}
we have
\begin{align*}
\abs{\frac{1}{\tilde f\big(\vec z_S(t)\big)}\!-\!\frac{1}{f\big(\vec z_S(t)\big)}}^p\!\!\!\!\leq\! 2^{p-1}\!\!\left(\!\frac{\abs{\tilde f\big(\vec z_S(t)\big)\!- \!f\big(\vec z_S(t)\big)}^{p}}{F_0^{2pM_t}}\!+\!\frac{\abs{\tilde f\big(\vec z_S(t)\big)\!-\!f\big(\vec z_S(t)\big)}^{2p}}{F_0^{2pM_t} a_n^p}\!\right).
\end{align*}
To control the conditional expectation of this quantity we need to control \[\EE_\Ss\abs{\tilde f\big(\vec z_S(t)\big)-f\big(\vec z_S(t)\big)}^q\text{ for }q\ge 1.\] We have
\begin{align*}
\EE_\Ss\abs{\tilde f_{S,\eta}\big(\vec z_S(t)\big)-f\big(\vec z_S(t)\big)}^{q}&\leq \EE_\Ss\abs{\hat f_{S,\eta}\big(\vec z_S(t)\big)-f\big(\vec z_S(t)\big)}^{q}+a_n^{q}\\
&\leq 2^{q-1}\EE_\Ss\abs{\hat f_{S,\eta}\big(\vec z_S(t)\big)-\EE_\Ss \hat f_{S,\eta}\big(\vec z_S(t)\big)}^q\\
&\qquad+2^{q-1}\abs{\EE_\Ss \hat f_{S,\eta}\big(\vec z_S(t)\big)-f\big(\vec z_S(t)\big)}^q+a_n^q.
\end{align*}
On the one hand, as $\EE_\Ss \hat f_{S,\eta}\big(\vec z_S(t)\big)=(H*f)\big(\vec z_S(t)\big)$, applying Lemma~\ref{lemma:WandJones} with kernel $H$ and function $f$ gives
\begin{align*}
\abs{\EE_\Ss \hat f_{S,\eta}\big(\vec z_S(t)\big)-f\big(\vec z_S(t)\big)}&\leq \frac{dM_t}{2}\eta^2 Q_{M_t} \int_{\RR^{dM_t}} \norm{z}_2^2 H(z)\dd z\\
&\leq \frac{1}{2}Q_{M_t}( dM_t\eta)^2 \int_\RR z^2 \Hh(z)\dd z.
\end{align*}
On the other hand define
\begin{align*}
\zeta_k=H_\eta\left(\vec z_S(t)-\vec Z_S^k(t)\right)-\EE_\Ss H_\eta\left(\vec z_S(t)-\vec Z_S^k(t)\right).
\end{align*}
Then
\begin{align*}
\EE_\Ss \abs{\hat f_{S,\eta}\big(\vec z_S(t)\big)-\EE_\Ss \hat f_{S,\eta}\big(\vec z_S(t)\big)}^q=\EE_\Ss\abs{\frac{1}{n}\sum_{k=1}^n \zeta_k}^q.
\end{align*}
The $\zeta_k$'s are conditional {\it i.i.d.} random variable given $\Ss$ so that for $q\le 2$, Jensen inequality gives
\begin{align*}
\EE_\Ss \abs{\frac{1}{n}\sum_{k=1}^n \zeta_k}^q\le \frac{F_\infty^{qM_t/2}\norm{\Hh}_2^{qdM_t}}{(n\eta^{dM_t})^{q/2}}.
\end{align*}
For $q>2$, by Khintchine inequality \citep[see][]{Bretagnolle1979} 
\begin{align*}
\EE_\Ss \abs{\frac{1}{n}\sum_{k=1}^n \zeta_k}^q\le \K_q\left(\frac{F_\infty^{qM_t/2}\norm{\Hh}_2^{qdM_t}}{(n\eta^{dM_t})^{q/2}}+\frac{2^{q-2}F_\infty^{M_t}(\norm{\Hh}_2^2 \norm{\Hh}_\infty^{q-2})^{dM_t}}{(n\eta^{dM_t})^{q-1}}\right),
\end{align*}
where $\K_q$ is the global constant in Khintchine inequality depending only on $q$.
Lemma follows.
\end{proof}

\begin{lemma}\label{lemma:II}
Under the assumptions of Theorem~\ref{theo:MSEcond},
\begin{align*}
\EE_\Ss B^2&\leq \Clemma{3}\Bigg(\frac{1}{n\eta^{dM_t}} +  h^4+\eta^4 + (n\eta^{dM_t})^{2\varepsilon-2} +\frac{1}{(n\eta^{dM_t})^{2\varepsilon}}\\
&\qquad + \frac{1}{(n\eta^{dMt})^{2\varepsilon+1}} + \frac{(n\eta^{dMt+4})^2}{(n\eta^{dM_t})^{2\varepsilon}} + \frac{1}{(n\eta^{dM_t})^2}+\frac{1}{(n\eta^{dM_t})^{3}}+h^8+\eta^8\\
&\qquad + (n\eta^{dM_t})^{4\varepsilon-4} + \frac{1}{(n\eta^{dM_t})^{4\varepsilon}} + \frac{1}{(n\eta^{dMt})^{4\varepsilon+3}} + \frac{(n\eta^{dM_t+4})^4}{(n\eta^{dM_t})^{4\varepsilon}}  \Bigg).
\end{align*}
where $\Clemma{3}$ depends only on $d$, $M_t$, $Q_{M_t}$, $F_0$, $F_\infty$, $\Theta$, $\norm{K}_2$, $\norm{\Hh}_2$, $\norm{\Hh}_\infty$, $\int_{\RR} z^2\Hh(z)\dd z$ and $\int_{\RR} u^2K(u)\dd u$.
\end{lemma}

\begin{proof}
Since
\begin{align*}
\frac{\tilde\phi\big(t,\vec z_S(t)\big)}{\tilde f\big(\vec z_S(t)\big)}-\frac{\phi\big(t,\vec z_S(t)\big)}{f\big(\vec z_S(t)\big)}=\frac{\tilde \phi\big(t,\vec z_S(t)\big)-\phi\big(t,\vec z_S(t)\big)}{f\big(\vec z_S(t)\big)}+\frac{\tilde \phi\big(t,\vec z_S(t)\big)}{\tilde f\big(\vec z_S(t)\big)}-\frac{\tilde \phi\big(t,\vec z_S(t)\big)}{f\big(\vec z_S(t)\big)},
\end{align*}
we have
\begin{align*}
\EE_\Ss B^2&\!\le\! 2\Bigg(\!\frac{\EE_\Ss\abs{\tilde\phi\big(t,\vec z_S(t)\big)-\phi\big(t,\vec z_S(t)\big)}^2}{F_0^{2M_t}}+(\Theta F_\infty)^{2M_t} \EE_\Ss\abs{\frac{1}{\tilde f\big(\vec z_S(t)\big)}-\frac{1}{f\big(\vec z_S(t)\big)}}^2\\
&\qquad+\frac{1}{2}{\EE_\Ss\abs{\tilde \phi\big(t,\vec z_S(t)\big)-\phi\big(t,\vec z_S(t)\big)}^4}+\frac{1}{2}{\EE_\Ss\abs{\frac{1}{\tilde f\big(\vec z_S(t)\big)}-\frac{1}{f\big(\vec z_S(t)\big)}}^4}\Bigg).
\end{align*}
Note that for $q=2,4$,
\begin{align*}
\EE_\Ss\abs{\tilde\phi_{S,h,\eta}\big(t,\vec z_S(t)\big)\!-\!\phi\big(t,\vec z_S(t)\big)}^q&\!\!\!\leq 2^{q-1}\Big(\EE_\Ss\abs{\tilde\phi_{S,h,\eta}\big(t,\vec z_S(t)\big)\!-\!\EE_\Ss\tilde \phi_{S,h,\eta}\big(t,\vec z_S(t)\big)}^q\\
&\qquad+\abs{\EE_\Ss \tilde \phi_{S,h,\eta}\big(t,\vec z_S(t)\big)\!-\!\phi\big(t,\vec z_S(t)\big)}^q\Big),
\end{align*}
where on the one hand, $\EE_\Ss\tilde\phi_{S,h,\eta}\big(t,\vec z_S(t)\big)=(\Kk_{h,\eta}*\psi)\big(t,\vec z_S(t)\big)$ with $\Kk$ the product kernel of $H$ and $K$ and 
\[
\forall (u,y)\in[0,t]\times\RR^{dM_t},\psi(u,y)=f(y)\theta_S\big(u,\vec y_S(u)\big).
\]
Applying Lemma~\ref{lemma:WandJones} with kernel $\Kk$ and function $\psi$ gives
\begin{align*}
\abs{\EE_\Ss \tilde \phi_{S,h,\eta}\big(t,\vec z_S(t)\big)-\psi\big(t,\vec z_S(t)\big)}&\leq \frac{1}{2}(dM_t+1)^2 C_{\Hh,K}Q_{M_t}\max(\eta^2,h^2),
\end{align*}
where $\psi\big(t,\vec z_S(t)\big)=\phi\big(t,\vec z_S(t)\big)$ and $C_{\Hh,K}=\max\left(\int_\RR z^2\Hh(z)\dd z,\int_\RR u^2K(u)\dd u\right)$ is a finite constant since $\Hh$ and $K$ are compactly supported.
On the other hand define
\begin{align*}
\xi_k = H_\eta\left(\vec z_S(t)-\vec Z_S^k(t)\right) \int_0^t K_h(t-s)\theta_S \left(s,\vec Z_S^k(s)\right)\dd s-\EE_\Ss\tilde\phi_{S,h,\eta}\big(t,\vec z_S(t)\big).
\end{align*}
Then 
\begin{align*}
\EE_\Ss\abs{\tilde\phi_{S,h,\eta}\big(t,\vec z_S(t)\big)-\EE_\Ss\tilde\phi_{S,h,\eta}\big(t,\vec z_S(t)\big)}^q=\EE_\Ss\abs{\frac{1}{n}\sum_{k=1}^n\xi_k}^q.
\end{align*}
The $\xi_k$ are conditional $i.i.d.$ centred random variables given $\Ss$ so that for $q\le 2$, by Jensen inequality 
\begin{align*}
\EE_\Ss\abs{\frac{1}{n}\sum_{k=1}^n\xi_k}^q\leq \frac{F_\infty^{qM_t/2}\Theta^{qM_t}\norm{\Hh}_2^{qdM_t}}{(n\eta^{dM_t})^{q/2}}.
\end{align*}
For $q> 2$, Khintchine inequality gives
\begin{align*}
\EE_\Ss \abs{\frac{1}{n}\sum_{k=1}^n \xi_k}^q \!\!\!\leq \K_q\! \left(\!\frac{F_\infty^{qM_t/2}\Theta^{qM_t}\norm{\Hh}_2^{qdM_t}}{(n\eta^{dM_t})^{q/2}} + \frac{2^{q-2}F_\infty^{M_t} \Theta^{qM_t} (\norm{\Hh}_2^2\norm{\Hh}_\infty^{q-2})^{dM_t}}{(n\eta^{dM_t})^{q-1}}\!\right).
\end{align*}
Using Lemma~\ref{lemma:I} to control $\EE_\Ss \abs{\frac{1}{\tilde f\big(\vec z_S(t)\big)}-\frac{1}{f\big(\vec z_S(t)\big)}}^p$ for $p= 2,4$, lemma follows.
\end{proof}

\begin{lemma}\label{lemma:phi}
Under the assumptions of Proposition~\ref{prop:CvProba}, almost surely
\begin{align*}
\EE_\Ss\hat \phi_{s,h,\eta}\big(t,\vec z_S(t)\big)&\xrightarrow[n\to+\infty]{}\phi\big(t,\vec z_S(t)\big)\text{ and }\\ nh\eta^{dM_t}\var_\Ss\hat\phi_{s,h,\eta}\big(t,\vec z_S(t)\big) &\xrightarrow[n\to+\infty]{}\phi\big(t,\vec z_S(t)\big)\norm{K}_2^2\norm{\Hh}_2^{2dM_t}.
\end{align*}
\end{lemma}

\begin{proof} We have
\begin{align*}
\EE_\Ss\hat\phi_{S,h,\eta}\big(t,\vec z_S(t)\big)=(\Kk_{h,\eta}*\psi)\big(t,\vec z_S(t)\big),
\end{align*}
with $\Kk$ the product kernel of $H$ and $K$ and \[\forall(u,y)\in[0,t]\times\RR^{dM_t},\psi(u,y)=f(y)\theta_S\big(u,\vec y_S(u)\big).\]
Applying Lemma~I.4 from \cite{Bosq1987}, with kernel $\Kk$ and function $\psi$ gives 
\begin{align}
\EE_\Ss \tilde \phi_{S,h,\eta}\big(t,\vec z_S(t)\big)\xrightarrow[n\to+\infty]{}\psi\big(t,\vec z_S(t)\big)=\phi\big(t,\vec z_S(t)\big).\label{eq:expectation}
\end{align}

Also, 
\begin{align*}
\EE_\Ss \hat \phi^2_{S,h,\eta}\big(t,\vec z_S(t)\big)&=\frac{1}{n}\EE_\Ss H^2_{\eta}\left(\vec z_S(t)-\vec Z_S^1(t)\right)\tilde \EE\left(\int_0^t K_h(t-s)\dd N_s^1\right)^2\\
&\qquad+\frac{n-1}{n}\EE_\Ss^2\hat\phi_{S,h,\eta}\big(t,\vec z_S(t)\big).
\end{align*}
As $N$ is a conditional Poisson process given $\Ss$ and $\Zz$,
\begin{align*}
\tilde \EE\left(\int_0^t K_h(t-s)\dd N_s^1\right)^2 &= \left(\int_0^t K_h(t-s)\theta_S\left(s,\vec Z_S^1(s)\right)\dd s\right)^2\\
&\qquad+ \int_0^t K_h^2(t-s)\theta_S\left(s,\vec Z_S(s)^1\right)\dd s.
\end{align*}
Define
\begin{align*}
\Kk_{h,\eta}^1(u,v,y)=K_h(u)K_h(v)H_\eta^2(y),\quad \Kk_{h,\eta}^2(u,y)=K_h^2(u)H_\eta^2(y)
\end{align*}
and $\psi_1$ a $(dM_t+2)$-variate function such that for all $(u,v,y)\in[0,t]^2\times\RR^{dM_t},$ $\psi_1(u,v,y)=f(y)\theta_S\big(u,\vec y_S(u)\big)\theta_S\big(v,\vec y_S(v)\big).$ Then
\begin{align*}
\var_\Ss \hat \phi^2_{S,h,\eta}\big(t,\vec z_S(t)\big)&=\frac{1}{nh\eta^{dM_t}}\left(\Kk_{h,\eta}^2*\psi\right)\big(t,\vec z_S(t)\big)\\
&\qquad+\frac{1}{n\eta^{dM_t}}\left(\Kk_{h,\eta}^1*\psi_1\right)\big(t,t,\vec z_S(t)\big) - \frac{1}{n}\EE_\Ss^2\hat \phi_{S,h,\eta}\big(t,\vec z_S(t)\big).
 \end{align*}
Applying Lemma~I.4 from \cite{Bosq1987}, with kernel $\Kk^1/\norm{\Hh}^{2dM_t}_2$ and function $\psi^1$, kernel $\Kk^2/\norm{K}^2_2\norm{\Hh}^{2dM_t}_2$ and function $\psi$ and using \eqref{eq:expectation} gives 
\begin{align*}
nh\eta^{dM_t}\var_\Ss\hat\phi_{s,h,\eta}\big(t,\vec z_S(t)\big) &\xrightarrow[n\to+\infty]{}\phi\big(t,\vec z_S(t)\big)\norm{K}_2^2\norm{\Hh}_2^{2dM_t}.
\end{align*}

\end{proof}

\begin{lemma}\label{lemma:phiDistr}
Under the assumptions of Theorem~\ref{theo:CvDistr},
\begin{align*}
(nh\eta^{dM_t})^{1/2}\frac{\hat\phi_{S,h,\eta}\big(t,\vec z_S(t)\big)- \EE_\Ss\hat\phi_{S,h,\eta}\big(t,\vec z_S(t)\big)}{\left(\phi\big(t,\vec z_S(t)\big) \norm{K}_2^2\norm{\Hh}_2^{2dM_t}\right)^{1/2}} \toDd\Nn(0,1).
\end{align*}
\end{lemma}

\begin{proof}
Denote
\begin{align*}
\hat L_k = (nh \eta^{d M_t})^{1/2}  \frac{ \displaystyle H_\eta \left(\vec z_S(t)-\vec Z_S^k(t)\right) \int_0^t K_h(t-s)\dd N_S^k-\EE_\Ss \hat\phi_{S,\eta}\big(t,\vec z_S(t)\big)}{n\left(\phi\big(t,\vec z_S(t)\big)\norm{K}_2^2\norm{\Hh}_2^{2dM_t}\right)^{1/2}},
\end{align*}
and
\begin{align*}
\hat L=(nh\eta^{dM_t})^{1/2}\frac{\hat\phi_{S,h,\eta}\big(t,\vec z_S(t)\big)- \EE_\Ss\hat\phi_{S,h,\eta}\big(t,\vec z_S(t)\big)}{\left(\phi\big(t,\vec z_S(t)\big) \norm{K}_2^2\norm{\Hh}_2^{2dM_t}\right)^{1/2}},
\end{align*}
so that $\hat L=\sum_{k=1}^n\hat L_k$.
The variables $(\hat L_1,\hat L_2,\ldots)$ are conditional independent random variables given $\Ss$ and
\begin{align*}
\forall k=1,\ldots,n, \EE_\Ss \hat L_k=0\text{ and }\var_\Ss\hat L=\sum_{k=1}^n \var_\Ss\hat L_k.
\end{align*}
To conclude it remains to check that the Lyapunov condition 
\begin{align} \label{eq:Lyapunov}
(\var_\Ss \hat L)^{-(2+\delta)/2} n\EE_\Ss \abs{\hat L_1}^{2+\delta}\to 0,
\end{align}
is satisfied for some $\delta>0$.
Remark that
\begin{align*}
\var_\Ss \hat L &= \frac{nh\eta^{dM_t}}{\phi\big(t,\vec z_S(t)\big)\norm{K}_2^2\norm{\Hh}_2^{2dM_t}} \var_\Ss \hat\phi_{S,h,\eta}\big(t,\vec z_S(t)\big),
\end{align*}
and Lemma~\ref{lemma:phi} gives
\begin{align*}
nh\eta^{dM_t}\var_\Ss\hat\phi_{s,h,\eta}\big(t,\vec z_S(t)\big) &\xrightarrow[n\to+\infty]{}\phi\big(t,\vec z_S(t)\big)\norm{K}_2^2\norm{\Hh}_2^{2dM_t},
\end{align*}
so that $\var_\Ss \hat L\to 1$ as $n\to+\infty$.
It suffices to show that $n\EE_\Ss (\lvert\hat L_1\rvert^{2+\delta})$ tends to $0$ for some $\delta>0$ to get the Lyapunov condition~\eqref{eq:Lyapunov} satisfied.
Let us take $\delta=2$, then
\begin{align*}
\EE_\Ss\abs{\hat L_1}^4&\leq \left(\frac{nh\eta^{dM_t}}{\phi\norm{K}_2^2\norm{\Hh}_2^{2dM_t}}\right)^2\frac{9}{n^4}\Bigg[\EE_\Ss \abs{H_\eta\left(\vec z_S(t)-\vec Z_S^1(t)\right)\int_0^t K_h(t-s)\dd \bar N_s^1}^4\\
&\qquad+\EE_\Ss\abs{H_\eta\left(\vec z_S(t)-\vec Z_S^1(t)\right)\int_0^t K_h(t-s)\theta_S\left(s,\vec Z_S^1(s)\right) \dd s}^4\\
&\qquad+\EE_\Ss^4 \hat\phi_{S,h,\eta}\big(t,\vec z_S(t)\big)\Bigg].
\end{align*}
Basic martingale properties as well as the Burkholder-Davis-Gundy inequality give
\begin{align*}
n\EE_\Ss\abs{\hat L_1}^4&\leq C\frac{h^2 \eta^{2dM_t}}{n}\left(\frac{1}{\eta^{dM_t}}+\frac{1}{\eta^{3dM_t}}+h^8+\eta^8+1\right),
\end{align*}
where the constant $C$ depends only on $\norm{\Hh}_2$, $\norm{K}_2$, $\int_\RR z^2\Hh(z)\dd z$, $\int_\RR z^2 K(z)\dd z$, $\norm{\theta_S}_\infty$, $F_\infty$, $Q_{M_t}$ and $dM_t$. We can conclude that $n\EE_\Ss(\hat L_1^4)\to 0$ as $n\to +\infty$.
Lemma follows.
\end{proof}

\section*{Acknowledgement}

\bibliography{IntensityCoxProcess}

\begin{thebibliography}{14}
\providecommand{\natexlab}[1]{#1}
\providecommand{\url}[1]{\texttt{#1}}
\expandafter\ifx\csname urlstyle\endcsname\relax
  \providecommand{\doi}[1]{doi: #1}\else
  \providecommand{\doi}{doi: \begingroup \urlstyle{rm}\Url}\fi

\bibitem[Asmussen and Albrecher(2010)]{Asmussen2010}
S.~Asmussen and H.~Albrecher.
\newblock \emph{Ruin probabilities}, volume~14.
\newblock World Scientific Publishing Company, 2010.

\bibitem[Bialek et~al.(1991)Bialek, Rieke, De~Ruyter~van Steveninck, and
  Warland]{Bialek1991}
W.~Bialek, F.~Rieke, R.~De~Ruyter~van Steveninck, and D.~Warland.
\newblock Reading a neural code.
\newblock \emph{Science}, 252\penalty0 (5014):\penalty0 1854--1857, 1991.

\bibitem[Biau et~al.(2010)Biau, C{\'e}rou, and Guyader]{Biau2010}
G.~Biau, F.~C{\'e}rou, and A.~Guyader.
\newblock Rates of convergence of the functional-nearest neighbor estimate.
\newblock \emph{Information Theory, IEEE Transactions on}, 56\penalty0
  (4):\penalty0 2034--2040, 2010.

\bibitem[Bickel(1982)]{Bickel1982}
Peter~J Bickel.
\newblock On adaptive estimation.
\newblock \emph{The Annals of Statistics}, pages 647--671, 1982.

\bibitem[Bosq and Lecoutre(1987)]{Bosq1987}
D.~Bosq and J.~P. Lecoutre.
\newblock \emph{Th{\'e}orie de l'estimation fonctionnelle}, volume~21.
\newblock Economica, 1987.

\bibitem[Bretagnolle and Huber(1979)]{Bretagnolle1979}
J.~Bretagnolle and C.~Huber.
\newblock Estimation des densit{\'e}s: risque minimax.
\newblock \emph{Probability Theory and Related Fields}, 47\penalty0
  (2):\penalty0 119--137, 1979.

\bibitem[Brette(2008)]{Brette2008}
R.~Brette.
\newblock Generation of correlated spike trains.
\newblock \emph{Neural computation}, 21\penalty0 (1):\penalty0 188--215, 2008.

\bibitem[Gy{\"o}rfi et~al.(2002)Gy{\"o}rfi, Kohler, Krzy{\.z}ak, and
  Walk]{Gyorfi2002}
L.~Gy{\"o}rfi, M.~Kohler, A.~Krzy{\.z}ak, and H.~Walk.
\newblock \emph{A distribution-free theory of nonparametric regression}.
\newblock Springer Science \& Business Media, 2002.

\bibitem[Kou et~al.(2005)Kou, Sunney~Xie, and Liu]{Kou2005a}
S.~C. Kou, X.~Sunney~Xie, and J.~S. Liu.
\newblock Bayesian analysis of single-molecule experimental data.
\newblock \emph{Journal of the Royal Statistical Society: Series C (Applied
  Statistics)}, 54\penalty0 (3):\penalty0 469--506, 2005.

\bibitem[Krumin and Shoham(2009)]{Krumin2009}
M.~Krumin and S.~Shoham.
\newblock Generation of spike trains with controlled auto-and cross-correlation
  functions.
\newblock \emph{Neural Computation}, 21\penalty0 (6):\penalty0 1642--1664,
  2009.

\bibitem[Merton(1976)]{Merton1976}
R.~C. Merton.
\newblock Option pricing when underlying stock returns are discontinuous.
\newblock \emph{Journal of financial economics}, 3\penalty0 (1):\penalty0
  125--144, 1976.

\bibitem[Ogata(1988)]{Ogata1988}
Y.~Ogata.
\newblock Statistical models for earthquakes occurences and residual analysis
  for point processes.
\newblock \emph{Journal of the Royal Statistical Society}, 44:\penalty0
  102--107, 1988.

\bibitem[O'Sullivan(1993)]{OSullivan1993}
F.~O'Sullivan.
\newblock Nonparametric estimation in the cox model.
\newblock \emph{The Annals of Statistics}, pages 124--145, 1993.

\bibitem[Zhang and Kou(2010)]{Zhang2010}
T.~Zhang and S.~C. Kou.
\newblock Nonparametric inference of doubly stochastic poisson process data via
  the kernel method.
\newblock \emph{The annals of applied statistics}, 4\penalty0 (4):\penalty0
  1913, 2010.

\end{thebibliography}

\end{document}